\documentclass[11pt, reqno]{amsart}

\usepackage{amssymb,amsfonts}
\usepackage{fullpage,url,amssymb,enumerate,colonequals}
\usepackage[all,arc]{xy}
\usepackage{enumerate}
\usepackage{enumitem}
\usepackage{mathrsfs}
\usepackage{hyperref}
\usepackage{mathtools}

\usepackage[colorinlistoftodos]{todonotes}
\usepackage{xcolor}
\usepackage{bbm}
\newtheorem{thm}{Theorem}[section]

\newtheorem{prop}[thm]{Proposition}
\newtheorem{lem}[thm]{Lemma}

\theoremstyle{definition}

\newtheorem{exmp}[thm]{Example}

\theoremstyle{remark}
\newtheorem{rem}[thm]{Remark}

\providecommand{\lb}{\left}
\providecommand{\rb}{\right}
\providecommand{\RR}{\mathbb{R}}

\providecommand{\ZZ}{\mathbb{Z}}

\providecommand{\CC}{\mathbb{C}}
\providecommand{\FF}{\mathbb{F}}
\providecommand{\Fmp}{\mathbb{F}_p^{\times}}
\providecommand{\PP}{\mathbb{P}}

\providecommand{\HH}{\mathbb{H}}

\providecommand{\ep}{\varepsilon}

\providecommand{\bk}{\backslash}

\providecommand{\cset}{\mathbbm{1}}

\DeclareMathOperator{\Gal}{Gal}

\DeclareMathOperator{\Res}{Res}

\DeclareMathOperator{\Tr}{Tr}
\DeclareMathOperator{\Det}{Det}

\DeclareMathOperator{\Ind}{Ind}

\DeclareMathOperator{\GL}{GL}
\DeclareMathOperator{\SL}{SL}
\DeclareMathOperator{\SO}{SO}
\DeclareMathOperator{\Aff}{Aff}
\DeclareMathOperator{\stab}{stab}
\DeclareMathOperator{\End}{End}
\DeclareMathOperator{\Nm}{Nm}

\usepackage[OT2,T1]{fontenc}
\DeclareSymbolFont{cyrletters}{OT2}{wncyr}{m}{n}
\DeclareMathSymbol{\sha}{\mathalpha}{cyrletters}{"58}

\makeatletter
\let\c@equation\c@thm
\makeatother
\numberwithin{equation}{section}

\title{A Selberg Trace Formula for $\text{GL}_{3}(\mathbb{F}_p)\backslash \text{GL}_{3}(\mathbb{F}_q)  /K$}

\begin{document}


\author{Daksh Aggarwal}
\address{
Department of Mathematics \\ 
Brown University\\ 
Providence, RI\\
USA \\
02912} 
\email{daksh\_aggarwal@brown.edu}

\author{Asghar Ghorbanpour}
\address{
Department of Mathematics \\ 
University of Western Ontario\\ 
London, ON\\
Canada \\
N6A 5B7} 
\email{aghorba@uwo.ca}

\author{Masoud Khalkhali}
\address{
Department of Mathematics \\ 
University of Western Ontario\\ 
London, ON\\
Canada \\
N6A 5B7} 
\email{masoud@uwo.ca}

\author{Jiyuan Lu}
\address{
Department of Mathematics \\ 
University of Toronto\\ 
Toronto, ON\\
Canada \\
M5S 2E4} 
\email{jiyuan.lu@mail.utoronto.ca}

\author{Bal\'azs N\'emeth}
\address{
DAMTP\\ 
University of Cambridge\\ 
Cambridge\\
UK\\
CB3 0WA} 
\email{bn273@cam.ac.uk}

\author{C Shijia Yu}
\address{
Department of Mathematics \\ 
University of Toronto\\ 
Toronto, ON\\
Canada \\
M5S 2E4} 
\email{cyu@math.toronto.edu}

\maketitle

\begin{abstract}
  In this paper, we prove a   discrete analog of the Selberg Trace Formula for the group $\text{GL}_{3}(\mathbb{F}_q).$ By considering a cubic extension of the finite field $\mathbb{F}_q$, we define an analog of the upper half space  and an action of $\text{GL}_{3}(\mathbb{F}_q)$  on it. To compute the orbital sums we explicitly identify the double coset spaces and fundamental domains in our upper half space. To understand the spectral side of the trace formula we decompose the induced representation $\rho = \text{Ind}_{\Gamma}^{G} 1$ for $G=  \text{GL}_{3}(\mathbb{F}_q) $ and $ \Gamma = \text{GL}_{3}(\mathbb{F}_p).$
\end{abstract}

\tableofcontents
\section{Introduction}

The Selberg trace     formula \cite{SelbergTraceFormula1956} is one  of the most celebrated  mathematical results  of the past century. In its original form, it relates the spectrum of the Laplacian on a  hyperbolic surface to lengths of closed geodesics on the surface. It can be viewed as a non-abelian generalization of the well-known Poisson summation formula in Fourier analysis and has found numerous applications in many fields, for example in number theory and the theory of automorphic forms (see Sarnak \cite{SARNAK1982229}).

The Selberg trace formula deals with  surfaces which can be represented as  double cosets of the form $\Gamma\bk \SL_2(\RR) /\SO_2(\RR)$, where $\Gamma$ is a discrete  subgroup of $\SL_2(\mathbb{R})$. It is a natural question to work out the trace formula for higher-dimensional matrix groups. To quote Dennis Hejhal: ``Of course, we won't really understand the trace formula until it is written down for $\SL_4(\ZZ)$.'' Dorothy Wallace \cite{wallaceSelbergTraceFormula1994} has worked out the explicit details of the trace formula for $\SL_3(\ZZ)\bk \SL_3(\RR)/\SO_3(\RR)$ but the four-dimensional case still remains mysterious. A further generalization is the so-called Arthur-Selberg trace formula \cite{arthurIntroductionTraceFormula}, which plays an important role in the Langlands program.

Selberg type trace formulae have been derived for discrete spaces like graphs as well. By realizing a $k$-regular graph as a quotient of the infinite $k$-regular tree, Ahumada \cite{ahumadaTraceFormula1987} found such a formula for $k$-regular graphs. Audrey Terras \cite{terras_1999} introduced the concept of finite upper half-planes as a finite analogue of the Poincar\'{e}  upper half-plane model where the complex number field is
replaced by a quadratic extension of a finite field $\FF_q$, where $q=p^r$ is a prime power. She constructed a family of Ramanujan graphs \cite{murtyRamanujanGraphs2003} using these upper-half planes and developed a representation-theoretic trace formula on $\GL_2(\FF_p)\bk \GL_2(\FF_q)/K$, where $K$ is the finite analogue of $\SO_2(\RR)$.

The aim of this paper is to present an analogous trace formula for $\GL_3(\FF_q)$. Our main result in this paper is  as follows.  Consider the  pre-trace formula 
\begin{equation}
  \label{eq:1}
  \sum_{\pi \in \hat{G}} m(\pi,\rho)\Tr(\hat{f}(\pi)) = \sum_{\{\gamma\} \in C_{\Gamma}} \frac{|G_\gamma|}{|\Gamma_\gamma|} I_G(f,\gamma),
\end{equation}
where  $G$ is a finite group,  $\Gamma$ is a subgroup of $G$,  
$\rho = \Ind_{\Gamma}^{G} 1$ is the induced representation of the trivial representation of $\Gamma$,  and $I_G(f,\gamma)$ is the orbital sum of $f$ at $\gamma$. These terms are defined in Section 3 and a proof of the pre-trace formula is recalled  there as well.  Note that (1.1) can be viewed as a representation-theoretic analogue of Selberg's trace formula \cite{SelbergTraceFormula1956,terras_1999}: the left hand side can be thought of as the spectral side, summing over the irreducible representations of $G$ and their multiplicities in $\rho$, while the right hand side can be interpreted as the geometric side, summing over conjugacy classes in $G$.
In this paper, to obtain a  trace formula for the finite upper half-space we take $G = \GL_3(\FF_q)$ and $\Gamma = \GL_3(\FF_p)$. Also, $f:G\to\CC$ is chosen to be a $K-$bi-invariant function ($K$ is the stabilizer of $p_0$ as above), i.e.\ $\forall k,h\in K$ and $\forall x\in G$ we have $f(kxh)=f(x)$. One can also think of $f$ as a function on $G/K\simeq\HH_q$ which is invariant under left multiplication by elements in $K$.

This paper  is organized as follows. In Section 2, we define an analogue, denoted $\HH_q,$ of the upper half plane for the group $\GL_3(\FF_q)$ by considering a cubic extension of the finite field $\FF_q$. We define an action of $\GL_3(\FF_q)$ on $\HH_q$ and this sets the stage for our discrete trace formula. In Section 3, we recall a standard pre-trace formula for finite groups. In Sections 4-6, we compute each term on the right hand side of the pre-trace formula \eqref{eq:4} (the so called geometric side) using the explicit description of conjugacy classes of $\GL_3(\FF_q)$ given in \cite{alaliConjugacyClassesGeneral1994}. The conjugacy classes are separated into central, hyperbolic, parabolic, and elliptic terms. Each of these types of terms are explicitly identified. It would be  helpful to have  a  more conceptual understanding  of  these computations. In Section 7, we give some simple examples of the application of our formula in which the left side of the pre-trace formula is known \textit{a priori}. In Sections 8-9, we calculate the left side of the pre-trace formula by computing the character of the induced representation $\rho = \Ind_{\Gamma}^G 1$ and then using that to compute its decomposition into irreducible representations in the case that $2,3\nmid n$. The calculation of the decomposition in the other cases is similar but involve more cases and hence  they are not  incorporated in  the arXiv version of this paper. They are available on demand. 

\subsection*{Acknowledgments}
\noindent
We would like to thank the Fields Institute for organizing the Fields Undergraduate Summer Research Program in 2021.  Daksh Aggarwal, Jiyuan Lu, Balazs N\'{e}meth,  and C Shijia Yu were participants in this program while working on this paper. The project was proposed and supervised by Masoud Khalkhali and Asghar Ghorbanpour.  

\section{A finite upper half-space}

To generalize Terras's  results to the $\GL_3(\FF_q)$ case, we must first discuss the generalization of the finite upper half-space. Let $p$ be a prime number and $q= p^n$ for   a positive integer $n$. The definition of the upper half space   $\HH_q$ comes from considering the cubic extension of $\FF_q$. The construction was first given by Martinez in \cite{martinezFiniteUpperHalfSpace2000}. Here we describe it explicitly for our special case of 
$\GL_3(\FF_q)$.

\begin{lem}
$\FF_q$ has a cubic nonresidue if and only if $q \equiv 1 \pmod{3}$. 
\end{lem}
\begin{proof}
   Let $q = p^n$. Consider the cube map $x\mapsto x^3: \FF_q^{\times} \to \FF_q^{\times}$. This map has as kernel the subgroup $\mu_3$ consisting of the cube roots of unity in $\FF_q^{\times}$. Let $\alpha$ be a generator of $\FF_q^{\times}$. Then if $q\equiv 1 \pmod{3}$, we have $\mu_3 = \{1,\alpha^{(q-1)/3},\alpha^{2(q-1)/3}\}$, in which case there are $(q-1)/3$ cubic residues.  Otherwise, $\mu_3 = \{1\}$, and so each element of $\FF_{q}^{\times}$ is a cubic residue.
\end{proof}

From now on, we assume that $q \equiv 1 \pmod{3}$. Take a cubic nonresidue $\delta \in \FF_q$ and consider the cubic extension $\FF_q\lb(\sqrt[3]{\delta}\rb) \simeq \FF_{q^3}$. A basis for $\FF_q\lb(\sqrt[3]{\delta}\rb)/\FF_q$ is $\{1,\delta^{1/3},\delta^{2/3}\}$. For an element $\alpha$, we will write $(\alpha_1,\alpha_2,\alpha_3)$ to denote its components, i.e.,
\[\alpha = \alpha_1 + \alpha_2 \delta^{1/3}+\alpha_3\delta^{2/3}.\]  We define 
\[\HH_q = \left\{(\alpha,\beta)\in \FF_q(\sqrt[3]{\delta})\times \FF_q(\sqrt[3]{\delta}):
  \begin{vmatrix}
    \alpha_2 & \alpha_3 \\
    \beta_2 & \beta_3
  \end{vmatrix} \neq 0
  \right\}.\]
\begin{lem}  The following map defines an action of $\GL_3(\FF_q)$ on $\HH_q:$
\[
  \begin{bmatrix}
    a & b & c\\
    d & e & f\\
    r & s & t
  \end{bmatrix} (\alpha,\beta) = \lb(\frac{a\alpha+b\beta+c}{r\alpha+s\beta+t}, \frac{d\alpha+e\beta+f}{r\alpha+s\beta+t} \rb).
\]
\end{lem}
\begin{proof}
 The compatibility of the action with multiplication in $\GL_3(\FF_q)$ follows from the fact that $\GL_3(\FF_q)$ has an action on the projective space $\PP^2(\FF_{q^3})$. So, we need to check that the action is closed in $\mathbb{H}_q$. This is a routine verification, which can be found in \hyperref[appA]{Appendix A}. 
\end{proof}

Our distinguished point will be $p_0 =(\delta^{2/3},\delta^{1/3}) \in \HH_q$. The action of $\Aff_3(\FF_q)\leq\GL_3(\FF_q)$ is transitive on $\HH_q$, where
\[\Aff_3(\FF_q) \coloneqq \lb\{
  \begin{bmatrix}
    a & b & c\\
    d & e & f\\
    0 & 0 & 1
  \end{bmatrix}
\rb\}.\]With this action, we calculate that
\[K \coloneqq \stab(p_0) = \lb\{
  \begin{bmatrix}
    a & c\delta & b\delta\\
    b & a & c\delta\\
    c & b & a
  \end{bmatrix}
\rb\}.\]
Indeed, we find that
\[K \xrightarrow{\sim} \FF_q(\delta^{1/3})^{\times}: \begin{bmatrix}
    a & c\delta & b\delta\\
    b & a & c\delta\\
    c & b & a
  \end{bmatrix} \mapsto a+b \delta^{1/3}+c \delta^{2/3}.\]
Also, we have the homogeneous space $\HH_q \simeq \GL_3(\FF_q)/K$. As a quick  check  note that we have $|\GL_3(\FF_q)| = (q^3-1)(q^3-q)(q^3-q^2)$ and $|K| = q^3-1$ (from the canonical isomorphism), so we should have $|\HH_q|$ equal to 
\[\frac{(q^3-1)(q^3-q)(q^3-q^2)}{q^3-1}=(q^2-1)(q-1)q^3.\]
Indeed, from the definition we have
\[|\HH_q| = |\GL_2(\FF_q)|\cdot q^2 = (q^2-1)(q^2-q)q^2=(q^2-1)(q-1)q^3.\]

\begin{rem}
Observe that the above construction is a natural generalization of the finite upper half-plane defined by Terras \cite{terras_1999} (here $\delta$ is a nonsquare in the finite field $\FF_q$):

$$H_q=\left\{x+y\sqrt{\delta}:x,y\in\FF_q, y\neq 0\right\}.$$
The condition $y\neq 0$ is equivalent to $1$ and $x+y\sqrt{\delta}$ being linearly independent elements of $\FF_q\lb(\sqrt\delta\rb)$ over $\FF_q$.

\end{rem}

\section{Trace formula}

\subsection{The pre-trace formula}

Our trace formula relies on interpretation of a known result, the pre-trace formula. We now derive the pre-trace formula and contextualize it in our setting of the group $G=\GL_3(\FF_q)$. The idea behind the pre-trace formula is to compute the trace of a specific operator in two different bases. We may construct the relevant information for this operator as follows: 

Suppose that $G$ is a finite group and $\Gamma\leq G$ is a subgroup. Let $f:G\to\CC$ be a complex-valued function on $G$. For a representation $\kappa:G\to \GL(V)$ of $G$ (where $V$ is a finite-dimensional complex vector space endowed with a linear $G$-action), define the Fourier transform of $f$ at $\kappa$ by:

$$\hat{f}(\kappa)=\sum_{g\in G}f(g)\kappa\lb(g^{-1}\rb).$$

Observe that the Fourier transform $\hat{f}(\kappa)$ is $\End V$-valued. The pre-trace formula is now obtained by computing the trace of $\hat{f}(\rho)$ in two different bases with $\rho=\Ind_\Gamma^G1$. First, we can decompose $\rho$ as a direct sum of irreducible representations of $G$ to get:

\begin{equation}
  \label{eq:1}
  \Tr\lb(\hat{f}(\rho)\rb)=\sum_{\pi\in\hat{G}}m(\pi,\rho)\Tr\lb(\hat{f}(\pi)\rb).
\end{equation}

Here $\hat{G}$ is the set of irreducible representations of $G$ and $m(\pi,\rho)$ is the multiplicity of the irreducible representation $\pi\in\hat{G}$ in $\rho$. Let\[
V=\{f:G\to\CC\,|\,f(\gamma g)=f(g)\quad\forall \gamma\in\Gamma, g\in G\}=L^2(\Gamma\bk G).
\]As $G$ acts on linearly on $V$ by right multiplication, the corresponding representation is precisely $\rho=\Ind_\Gamma^G1$. More concretely, for $\phi\in V$, $g,x\in G$ we have

$$\lb[\rho(g)\phi\rb](x)=\phi(xg).$$

Thus:

$$\lb[\hat{f}(\rho)\phi\rb](x)=\sum_{y\in G}f(y)\phi\lb(xy^{-1}\rb)=\sum_{u\in G}f\lb(u^{-1}x\rb)\phi(u).$$

Further, as $\phi$ is $\Gamma$-invariant,

$$\lb[\hat{f}(\rho)\phi\rb](x)=\sum_{y\in\Gamma\bk G}\sum_{\gamma \in \Gamma}f\lb(y^{-1}\gamma x\rb)\phi(y).$$

Now we are ready to compute the trace of $\hat{f}(\rho)$ in another way: consider the indicator basis $\lb\{\delta_y\rb\}_{y\in\Gamma\bk G}$ of $V$\footnote{For $y\in\Gamma\bk G$, $\delta_y$ is the function on $\Gamma\bk G$ which outputs $1$ if the input is $y$ and $0$ otherwise.}; with respect to this basis, $\hat{f}(\rho)$ is represented by a matrix where the entry corresponding to $x,y\in\Gamma\bk G$ is $\sum_{\gamma\in\Gamma}f\lb(y^{-1}\gamma x\rb)$. Therefore,

\begin{equation}
  \label{eq:2}
  \Tr\lb[\hat{f}(\rho)\rb]=\sum_{x\in\Gamma\bk G}\sum_{\gamma\in\Gamma}f\lb(x^{-1}\gamma x\rb).
\end{equation}

We can rewrite the right-hand side as a sum over conjugacy classes in $G$. Let:\[
\Gamma_\gamma=\left\{x\in\Gamma\,|\,x^{-1}\gamma x=\gamma\right\}=\text{centralizer of }\gamma \text{ in }\Gamma
\]\[
G_\gamma=\left\{x\in G\,|\,x^{-1}\gamma x=\gamma\right\}=\text{centralizer of }\gamma \text{ in }G
\]\[
\{\gamma\}=\left\{x^{-1}\gamma x\,|\, x\in G\right\}=\text{conjugacy class of }\gamma \text{ in }G
\]\[
C_\Gamma=\text{set of conjugacy classes in }\Gamma.
\]

Note that there is a 1-to-1 correspondence between right cosets $\Gamma_\gamma\backslash\Gamma$ and elements of $\{\gamma\}$ given by $\Gamma_\gamma x\mapsto x^{-1}\gamma x$. Since conjugacy classes partition $\Gamma$ we can rewrite the inner sum as:

$$\sum_{x\in\Gamma\backslash G}\sum_{\gamma\in\Gamma}f(x^{-1}\gamma x)=\sum_{x\in\Gamma\backslash G}\sum_{\{\gamma\}\in C_\Gamma}\sum_{u\in\Gamma_\gamma\backslash\Gamma}f\left(x^{-1}u^{-1}\gamma ux\right).$$

Set $y=ux$, then we can rewrite this as a sum over right cosets $y\in\Gamma_\gamma\backslash G$:

$$\sum_{x\in\Gamma\backslash G}\sum_{\gamma\in\Gamma}f(x^{-1}\gamma x)=\sum_{y\in\Gamma_\gamma\backslash G}\sum_{\{\gamma\}\in C_\Gamma}f(y^{-1}\gamma y).$$

Observe that $\Gamma_\gamma\leq G_\gamma$ hence we can write $y=ts$ for $t\in\Gamma_\gamma\backslash G_\gamma$ and $s\in G_\gamma\backslash G$:

\begin{equation}
  \label{eq:3}
  \sum_{x\in\Gamma\backslash G}\sum_{\gamma\in\Gamma}f(x^{-1}\gamma x)=\sum_{t\in\Gamma_\gamma\backslash G_\gamma}\sum_{s\in G_\gamma\backslash G}\sum_{\{\gamma\}\in C_\Gamma}f\left(s^{-1}t^{-1}\gamma ts\right)=\sum_{\{\gamma\}\in C_\Gamma}\frac{|G_\gamma|}{|\Gamma_\gamma|}\sum_{s\in G_\gamma\backslash G}f(s^{-1}\gamma s).
\end{equation}

Define the orbital sum $I_G(f,\gamma)$ of $f$ at $\gamma$ to be:

$$I_G(f,\gamma)\coloneqq \sum_{s\in G_\gamma\backslash G}f(s^{-1}\gamma s).$$

The pre-trace formula now follows from combining \eqref{eq:1}$-$\eqref{eq:3}:
\begin{equation}
  \label{eq:4}
  \sum_{\pi \in \hat{G}} m(\pi,\rho)\Tr(\hat{f}(\pi)) = \sum_{\{\gamma\} \in C_{\Gamma}} \frac{|G_\gamma|}{|\Gamma_\gamma|} I_G(f,\gamma),
\end{equation}
where $\rho = \Ind_{\Gamma}^{G} 1$ and $I_G(f,\gamma)$ is the orbital sum of $f$ at $\gamma$ as defined above. Note that \eqref{eq:4} can be viewed as a representation-theoretic analogue of Selberg's trace formula \cite{SelbergTraceFormula1956,terras_1999}: the left hand side can be thought of as the spectral side, summing over the irreducible representations of $G$ and their multiplicities in $\rho$, while the right hand side can be interpreted as the geometric side, summing over conjugacy classes in $G$. For further reference see \cite[Ch.\ 22]{terras_1999}.

To obtain the trace formula for the finite upper half-space we take $G = \GL_3(\FF_q)$ and $\Gamma = \GL_3(\FF_p)$. Also, $f:G\to\CC$ is chosen to be a $K-$bi-invariant function ($K$ is the stabilizer of $p_0$ as above), i.e.\ $\forall k,h\in K$ and $\forall x\in G$ we have $f(kxh)=f(x)$. One can also think of $f$ as a function on $G/K\simeq\HH_q$ which is invariant under left multiplication by elements in $K$.

\begin{rem}
    The left-hand side of \eqref{eq:4} can be expressed solely in terms of the characters $\chi_\pi$ of the group $G$:
    \begin{equation*}
    \begin{split}
        \Tr\lb[\hat{f}(\pi)\rb]&=\Tr\lb[\sum_{g\in G}f(g)\pi\lb(g^{-1}\rb)\rb]=\sum_{g\in G}f(g)\Tr\lb[\pi\lb(g^{-1}\rb)\rb]=\\&=\sum_{g\in G}f(g)\chi_\pi\lb(g^{-1}\rb)=\sum_{g\in G}f(g)\overline{\chi_\pi(g)}=\langle f,\chi_\pi\rangle.
    \end{split}
    \end{equation*}
    
    Here $\langle x,y\rangle$ denotes the usual Hermitian inner product on the complex vector space $L^2(G)$:
    
    $$\langle x,y\rangle=\sum_{g\in G}x(g)\overline{y(g)}$$
\end{rem}

We will, in Sections 4-6, compute each term in the sum on the right hand side of the pre-trace formula \eqref{eq:4} using the explicit description of conjugacy classes of $\GL_3(\FF_p)$ given in \cite{alaliConjugacyClassesGeneral1994}. The conjugacy classes are separated into the cases of central, hyperbolic, parabolic, and elliptic terms. It is helpful to introduce one more tool that will help us simplify these computations.

\subsection{Double cosets and fundamental domains}

To calculate the orbital sums it will be convenient to identify $G_\gamma\bk G$ with $G_\gamma\bk G/K\times K/Z$ (when possible), where $Z$ is the center of $G$, the subgroup of diagonal matrices $\{aI:a\in\FF_q^\times\}$. Let $\{t_i\}_{i=1}^l$ be representatives of double cosets in $G_\gamma\bk G/K$ and consider the map:

\begin{equation*}
	\begin{split}
		G_\gamma\bk G/K\times K/Z & \to G_\gamma\bk G\\
		(G_\gamma t_iK,kZ) & \mapsto G_\gamma t_ik.
	\end{split}
\end{equation*}

We first show that this map is well-defined and onto. If $z\in Z$ then $G_\gamma t_ikz=G_\gamma zt_ik=G_\gamma t_ik$ because $z\in G_\gamma$ as $z$ commutes with every element of $G$. Now let $G_\gamma x\in G\bk G_\gamma$, then since double cosets $G_\gamma t_iK$ partition $G$ and they are unions of right cosets $G_\gamma y$ we can find a representative $t_i\in G$ and $k\in K$ such that $G_\gamma x=G_\gamma t_ik$.

As for injectivity, let $t_i,t_j$ be representatives of two double cosets in $G_\gamma\bk G/K$ and $k,k'$ be representatives of two left cosets in $K/Z$. Suppose that $G_\gamma t_ik=G_\gamma t_jk'.$ Since double cosets are either disjoint or identical this is only possible if $t_i=t_j$, let $t=t_i$. So, we have,
 \begin{equation}
  \label{eq:6}
  G_\gamma=G_\gamma tk'k^{-1}t^{-1},
\end{equation}
which happens if and only if $tk'k^{-1}t^{-1}\in G_\gamma$, equivalently, if $k'k^{-1}\in tG_\gamma t^{-1}\cap K$. This can also be stated as $kH=k'H$ where $H=tG_{\gamma}t^{-1}\cap K$. Note that $Z\leq K$ and $Z\leq tG_\gamma t^{-1}=G_{t\gamma t^{-1}}$ as $Z$ is the intersection of all centralizers in $G$, hence we always have $Z\leq K\cap tG_\gamma t^{-1}=H$. If $Z=H$, then \eqref{eq:6} is equivalent to $kZ=k'Z$, which implies that the above map is injective. In our specific example, the centralizers $G_\gamma$ contain no non-central elements conjugate to some matrix in $K$ unless $\gamma$ itself is conjugate to a member of $K$. (This can be seen e.g.\ by taking simple representatives $\gamma$ of conjugacy classes in $G$ and looking at the characteristic polynomials and eigenvalues of elements in $G_\gamma$, which are preserved by conjugation.)

Thus we have two cases:

\textit{Case 1:} $\gamma$ is not conjugate to an element of $K$.   Then we have:

$$I_G(f,\gamma)=\sum_{s\in G_\gamma\bk G}f(s^{-1}\gamma s)=\sum_{t\in G_\gamma\bk G/K}\sum_{k\in K/Z}f((tk)^{-1}\gamma tk)=|K/Z|\sum_{t\in G_\gamma\bk G/K}f(t^{-1}\gamma t).$$

(The last equality follows from the fact that $f$ is $K-$bi-invariant.) We can identify $G_\gamma\bk G/K$ with $G_\gamma\bk \HH_q$. In order to calculate orbital sums we will find fundamental domains for the action of $G_\gamma$ on $\HH_q$, i.e. subsets of $\HH_q$ which contain exactly one element from each $G_\gamma$-orbit.

\textit{Case 2:} $\gamma$ is conjugate to an element of $K$.
In this case, the above map is not necessarily injective. However, $G_\gamma=G\text{ or }K$ depending on whether $\gamma$ is conjugate to a central element of $K$ or not. Then either $G_\gamma\bk G$ is trivial or it can be identified with $\mathbb{H}_q$.

\section{Central and Hyperbolic Terms}

\subsection{Central terms}

Let $\gamma =
  \begin{bmatrix}
    a & 0 & 0\\
    0 & a & 0\\
    0 & 0 & a
  \end{bmatrix}$ for $a\in\FF_p^\times$. In this case, $G_{\gamma} = G$ and $\Gamma_{\gamma} = \Gamma$. So,
  $I_G(f,\gamma) = f(p_0)$ and the total contribution from such $\gamma$ is 
  \[\frac{|G|}{|\Gamma|}f(p_0)\cdot (p-1) = \frac{(q^3-1)(q^2-1)(q-1)q^3}{(p^3-1)(p^2-1)p^3}f(p_0).\]
  
\subsection{Hyperbolic terms of the first kind} 
Let $\gamma =
  \begin{bmatrix}
    a & 0 & 0\\
    0 & a & 0\\
    0 & 0 & b
  \end{bmatrix}, a\neq b\in\FF_p^\times
  .$ Here \[G_\gamma = \lb\{
    \begin{bmatrix}
      x & y & 0 \\
      w & z & 0 \\
      0 & 0 & t
    \end{bmatrix}
    \rb\} \simeq \GL_2(\FF_q) \times \FF_{q}^{\times}\]
  and similarly for $\Gamma_{\gamma}$.

A fundamental domain for the action of $G_{\gamma}$ is given by
  \begin{align}
    \label{eq:5}
    G_\gamma\bk \HH_q = \{(u+\delta^{1/3},v+\delta^{2/3}):u,v\in \FF_q\}.
  \end{align}
  We verify this in the next proposition for the sake of illustration. In the next sections, the verificiation of the fundamental domains is relegated to \hyperref[appA]{Appendix A}.
  \begin{prop}
    A fundamental domain for $G_{\gamma}$ is given by \eqref{eq:5}.
  \end{prop}
\begin{proof}
First to show the uniqueness of each representative, suppose there exists $m\in G_{\gamma}$ such that
      \[m\cdot (u_1+\delta^{1/3},v_1  + \delta^{2/3}) = (u_2+ \delta^{1/3}, v_2 + \delta^{2/3}).\]
      Then we deduce that $y = 0 = w$, and $x=t=z$. So, $u_1 = u_2$ and $v_1 = v_2$.
      
       Next, to show the completeness of this fundamental domain, take an arbitrary element $(\alpha,\beta) \in \HH_q$. We wish to find $m\in G_{\gamma},u,v\in \FF_q$ such that $m\cdot (u+ \delta^{1/3}, v+\delta^{2/3}) = (\alpha,\beta)$. Set $x = \alpha_2, y = \alpha_3, w = \beta_2,z =\beta_3$, and $t=1$. Then we see $u,v$ is a solution to
       \begin{align*}
          \begin{bmatrix}
        \alpha_2 & \alpha_3\\
        \beta_2 & \beta_3\\
      \end{bmatrix}
         \begin{bmatrix}
           u\\
           v
         \end{bmatrix}
         =
         \begin{bmatrix}
           \alpha_1\\
           \beta_1
         \end{bmatrix},
       \end{align*}
       which has a solution since the left matrix is invertible.
  \end{proof}
  
  Thus,
  \begin{align*}
    I_G(f,\gamma) &= \sum_{t\in G_{\gamma}\bk G/K} \sum_{u\in K/\{aI\}}  f((tu)^{-1}\gamma (tu))\\
                  & = \sum_{t\in G_{\gamma}\bk \HH_q} \frac{|K|}{|\{aI\}|} f(t^{-1}\gamma t)\\
    & = \frac{q^3-1}{q-1}  \sum_{x,y\in \FF_q} f\lb(
      \begin{bmatrix}
        0 & 1 & -y\\
        1 & 0 & -x\\
        0 & 0 & 1
      \end{bmatrix}\gamma
               \begin{bmatrix}
        0 & 1 & x\\
        1 & 0 & y\\
        0 & 0 & 1
      \end{bmatrix} p_0
                \rb) \\
    & = (q^2+q+1) \sum_{x,y\in \FF_q}f\lb( \begin{bmatrix}
        a & 0 & (a-b)y\\
        0 & a & (a-b)x\\
        0 & 0 & b
      \end{bmatrix} p_0\rb) \\
    & = (q^2+q+1) \sum_{x,y\in \FF_q}f\lb( (x+\frac{a}{b}\delta^{2/3}, y+\frac{a}{b}\delta^{1/3})\rb).\\
  \end{align*}
  Let us define the \emph{horocycle transform} $Hf:\GL_2(\FF_q)\to\mathbb{C}$ of $f\in L^2(K\bk G/K)$ by
  \[Hf(\kappa)\coloneqq \sum_{x,y\in\FF_q}\sum_{\xi\in\{\kappa\}}f\lb(\begin{bmatrix}\xi_{11}&\xi_{12}&x\\ \xi_{21}&\xi_{22}&y\\ 0&0&1\end{bmatrix}p_0\rb),\]
  
  where the inner sum runs over $\xi\in\GL_2(\FF_q)$ in the conjugacy class of $\kappa$.
So, this hyperbolic term is equal to 
  \begin{align*}
    & (q^2+q+1) \sum_{\substack{a, b\in \FF_p^{\times}\\a\neq b}} \frac{(q^2-1)(q^2-q)(q-1)}{(p^2-1)(p^2-p)(p-1)}\cdot Hf\lb(\frac{a}{b}I\rb)\\
    & = \frac{(q^3-1)(q^2-1)(q^2-q)}{(p^2-1)(p^2-p)(p-1)}\sum_{\substack{a\in \FF_p^{\times}\\a\neq 1}}(p-1)\cdot Hf(aI)\\
    & =  \frac{(q^3-1)(q^2-1)(q^2-q)}{(p^2-1)(p^2-p)}\sum_{\substack{a\in \FF_p^{\times}\\a\neq 1}}Hf(aI).\\
  \end{align*}
We have shown the full computation for the orbital sums for this case both to demonstrate the computational technique and to motivate and introduce the horocycle transform. For the next cases we omit detailed computations for orbital sums, but included some in \hyperref[appA]{Appendix A} for the curious reader.

 \subsection{Hyperbolic terms of the second kind}
 Let $\gamma =
    \begin{bmatrix}
      a & 0 & 0\\
      0 & b & 0 \\
      0 &  0 & c
    \end{bmatrix}$ with $a,b,c\in\FF_p^\times$ distinct. Here $G_\gamma$ is the set of diagonal matrices and:

\begin{align*}
    G_\gamma\bk \HH_q &= \{(x+\delta^{1/3}+y\delta^{2/3}, r+s\delta^{1/3}+\delta^{2/3}): ys\neq 1\}\sqcup\{(x+\delta^{2/3},r+\delta^{1/3}+s\delta^{2/3}):s\neq 0\}\\
    &\qquad \sqcup\{(x+y\delta^{1/3}+\delta^{2/3},r+\delta^{1/3}):y\neq 0\}\sqcup\{(x+\delta^{2/3},y+\delta^{1/3})\}
\end{align*}
We have\[
I_G(f,\gamma)=(q^2+q+1)\cdot Hf\lb(\begin{bmatrix}
 a/c & 0\\
 0 & b/c
\end{bmatrix}\rb),
\] where $a,b,c\in \FF_p^{\times}$ comes from a change of variables outlined in \hyperref[appA]{Appendix A}.

The  total contribution from the second hyperbolic term is:
  $$\frac{1}{6}(q^2+q+1) \frac{(q-1)^3}{(p-1)^3}\sum_{\substack{a,b,c\in \FF_p^{\times}\\a\neq b\neq c}}Hf\lb(\begin{bmatrix}
  a/c&0\\0&b/c\end{bmatrix}\rb)=\frac16(q^2+q+1)\frac{(q-1)^3}{(p-1)^2}\sum_{\substack{a,b\in\FF_p^\times\backslash\{1\}\\a\neq b}}Hf\lb(\begin{bmatrix}
  a&0\\0&b\end{bmatrix}\rb)=$$
  
  $$=\frac13(q^3-1)\frac{(q-1)^2}{(p-1)^2}\sum_{\substack{\{a,b\}\subseteq\FF_p^\times\backslash\{1\}}}Hf\lb(\begin{bmatrix}
  a&0\\0&b\end{bmatrix}\rb).$$

\section{Parabolic Terms}

\subsection{Parabolic terms of the first kind}
Let $\gamma =
    \begin{bmatrix}
      a & 0 & a\\
      0 & a & 0 \\
      0 & 0 & a
    \end{bmatrix}, a \in \FF_p^{\times}$. In this case,
    \[G_\gamma = \lb\{
      \begin{bmatrix}
        d & y & x\\
        0 & c & b\\
        0 & 0 & d
      \end{bmatrix}
      :c,d\in\FF_q^\times,b,x,y\in\FF_q\rb\}.\]
    \begin{prop}\label{lemma:G-fundamental-domain}
      A fundamental domain for $G_{\gamma}$ is
       \[G_{\gamma} \bk \HH_q= \{(u\delta^{1/3},v\delta^{1/3}+\delta^{2/3}): u,v\in \FF_q, u\neq 0\} \sqcup  \{(\delta^{1/3}+ u\delta^{2/3}, \delta^{1/3}): u\in \FF_q^{\times}\} .\]
     \end{prop}
     \begin{proof}
       See \hyperref[appA]{Appendix A}.
     \end{proof}
   
    


   Following the computations outlined in \hyperref[appA]{Appendix A}, the orbital sum in this case is
     \[I_G(f,\gamma)=(q^2+q+1)\lb(Hf(I)-f(p_0)\rb).\]


And in total, the contribution from the parabolic terms of the first kind is
\[\frac{q^3(q-1)^2}{p^3(p-1)^2}(q^2+q+1)\lb(Hf(I)-f(p_0)\rb)\cdot (p-1) = \frac{q^3(q^3-1)(q-1)}{p^3(p-1)}\lb(Hf(I)-f(p_0)\rb).\]

\subsection{Parabolic terms of the second kind}
Let $\gamma =
    \begin{bmatrix}
      a & a & 0\\
      0 & a & a\\
      0 & 0 & a
    \end{bmatrix}, a \in \FF_p^{\times}$. We find that
    \[G_\gamma = B = \lb\{
      \begin{bmatrix}
        z & y & x\\
        0 & c & b\\
        0 & 0 & d
      \end{bmatrix}:d,c,z\in\FF_q^\times,b,x,y\in\FF_q
      \rb\}.\]
    Here $B$ is the Borel subgroup of upper-triangular matrices.
    \begin{prop}\label{lemma:B-fundamental-domain}
      A fundamental domain for $B$ is
       \[B \bk \HH_q= \{(\delta^{1/3},v\delta^{1/3}+\delta^{2/3}): v\in \FF_q\} \sqcup  \{(\delta^{2/3},\delta^{1/3})\} .\]
     \end{prop}
     \begin{proof}
       See \hyperref[appA]{Appendix A}.
     \end{proof}

     We find the orbital sum (see \hyperref[appA]{Appendix A}) is
     \[
I_G(f,\gamma)=(q^2+q+1)\lb(f\lb(\delta^{2/3}+\delta^{1/3},\delta^{1/3}+1\rb)+\sum_{v\in \FF_q} f\lb((1-v)\delta^{2/3}-v^2\delta^{1/3}+1,\delta^{2/3}+(v+1)\delta^{1/3}\rb)\rb).
\]
Note that the orbital sum is independent of $a$, and so we denote each $I_G(f,\gamma)$ by $\tilde{I}_G(f)$. Thus, in total, the contribution of these parabolic terms is \[
\frac{q^3(q-1)^2}{p^3(p-1)^2}\tilde{I}_G(f)\cdot(p-1)=\frac{q^3(q-1)^2}{p^3(p-1)}\tilde{I}_G(f).
\]

\subsection{Parabolic terms of the third kind}
Let $\gamma =  \begin{bmatrix}
      a & 0 & 0\\
      a & a & 0\\
      0 & 0 & b
    \end{bmatrix}$, with $a\neq b\in\FF_p^\times$. Here
 \[G_\gamma = \lb\{
      \begin{bmatrix}
        x & 0 & 0\\
        y & x & 0\\
        0 & 0 & z
      \end{bmatrix}: x,z \in \FF_{q}^{\times}, y\in \FF_q\rb\}.\]

    Again, via a change of variables, we find the orbital sum (see \hyperref[appA]{Appendix A})
    \[
I_G(f,\gamma)=(q^2+q+1)\cdot Hf\lb(\begin{bmatrix}
  a/b&1\\0&a/b
\end{bmatrix}\rb).
\]

The overall contribution of these terms is
\begin{align*}
  \sum_{\substack{a,b\in\FF_p^\times\\a\neq b}}\frac{q(q-1)^2}{p(p-1)^2}(q^2+q+1)\cdot Hf\lb(\begin{bmatrix}
  a/b&1\\0&a/b
\end{bmatrix}\rb)&=\frac{q(q-1)^2}{p(p-1)^2}(q^2+q+1)\sum_{\substack{c\in\FF_p^\times\\ c\neq 1}}(p-1)\cdot Hf\lb(\begin{bmatrix}
  c&1\\0&c
\end{bmatrix}\rb)\\
  & = \frac{q(q-1)(q^3-1)}{p(p-1)}\sum_{\substack{c\in\FF_p^\times\\c\neq 1}}Hf\lb(\begin{bmatrix}
  c&1\\0&c
\end{bmatrix}\rb).
\end{align*}

\section{Elliptic Terms}
\subsection{Elliptic terms of the first kind}
Irreducible $\gamma$. The characteristic polynomial of such a $\gamma$ is irreducible in $\FF_p[t]$. Recall that $q=p^n$. Depending on whether $n\equiv 0\pmod{3}$, there are two cases to consider.
\begin{itemize}
    \item $n\equiv 0 \pmod{3}$. In this case, we have the tower of field extensions $\FF_q/\FF_{p^3}/\FF_p$. So, $\gamma$ is similar to a diagonal matrix in $\GL_3(\FF_q)$. But that is not the case in $\GL_3(\FF_p)$, which is why this case is different from the second hyperbolic term. Therefore $G_{\gamma}$ is the subgroup of diagonal matrices over $\FF_q$ while $\Gamma_{\gamma}$ is $K$ with entries in $\FF_p$, i.e. $|\Gamma_{\gamma}|=p^3-1$. Since $G_{\gamma}$ is the same, we can reuse the fundamental domain for $G_{\gamma}$ from earlier, in the second hyperbolic case.
    
    Suppose $\gamma$ is similar to $\bigg(\begin{smallmatrix}\xi_1&0&0\\0&\xi_2&0\\0&0&\xi_3\\ \end{smallmatrix}\bigg)$ in $\GL_3(\FF_q)$, where $\xi_1,\xi_2,\xi_3\in\FF_{p^3}\bk\FF_p$. Similarly as for the second hyperbolic term the orbital sum will be:
    
    $$I_G(f,\gamma)=|K/Z|\cdot Hf\lb(\begin{bmatrix}\xi_1/\xi_3&0\\0&\xi_2/\xi_3\end{bmatrix}\rb).$$
       
    Since $\xi_1,\xi_2,\xi_3$ are Galois conjugates without loss of generality $\xi_1=\xi_3^p$ and $\xi_2=\xi_3^{p^2}$. The total contribution of the first elliptic terms is therefore:
    
    $$\frac13\cdot\frac{(q-1)^3}{p^3-1}\cdot\frac{q^3-1}{q-1}\sum_{\xi\in\FF_{p^3}^\times\bk\FF_p^\times}Hf\lb(\begin{bmatrix}\xi^{p-1}&0\\0&\xi^{p^2-1}\end{bmatrix}\rb).$$
    
    (The factor of $\frac13$ comes from the fact that each $\gamma$ is counted three times in the sum.)
    
    \item $n\not\equiv 0 \pmod{3}$. In this case, there is no cubic extension intermediate to $\FF_q/\FF_p$, and so $\FF_q(\delta^{1/3})$ is the minimal field (of course, in the sense of containment, not size) over which the characteristic polynomial of $\gamma$ has a root. Suppose an eigenvalue of $\gamma$ is $\alpha_1+\alpha_2\delta^{1/3}+\alpha_3\delta^{2/3}$ with eigenvector $x_1+x_2\delta^{1/3}+x_3\delta^{2/3}$, where $\alpha_i\in \FF_q$ and $x_i \in \FF_q^3$. Then, 
    \begin{align*}
        \gamma x_1 &= \alpha_1 x_1 + \alpha_3\delta x_2 +\alpha_2 \delta x_3,\\
         \gamma x_2 &= \alpha_2 x_1 + \alpha_1 x_2 +\alpha_3 \delta x_3,\\
          \gamma x_3 &= \alpha_3 x_1 + \alpha_2 x_2 +\alpha_1  x_3.
    \end{align*}
    Thus, viewed in $\GL_3(\FF_q),$ $\gamma$ is similar to a matrix in $K$. Conversely, non-diagonal matrices in $K$ give us irreducible elements in $\GL_3(\FF_p)$. Further, in this case $G_{\gamma}$ is precisely $K$, and so $G_{\gamma}\bk G = \HH_q$. After identifying elements $y\in \HH_q$ with elements $y=\begin{bmatrix}
      a&b&c\\d&e&f\\0&0&1
    \end{bmatrix}\in\Aff_3(\FF_q)$ and writing $\gamma=\begin{bmatrix}
    \alpha & \ep\delta & \beta\delta\\
    \beta & \alpha & \ep\delta\\
    \ep & \beta & \alpha
    \end{bmatrix}$, we can compute:
    \begin{align*}
        I_G(f,\gamma)&=\sum_{y\in \HH_q} f(y^{-1}\gamma y)\\
        &= \sum_{a,b,c,d,e,f\in \FF_q} f\lb(\frac{1}{bd-ae}\begin{bmatrix}-e & b & ce-bf\\
        d & -a & af-cd\\
        0 & 0 & bd-ae
        \end{bmatrix}\begin{bmatrix}
    \alpha & \ep\delta & \beta\delta\\
    \beta & \alpha & \ep\delta\\
    \ep & \beta & \alpha
    \end{bmatrix}\begin{bmatrix}
        a & b & c\\
        d & e & f\\
        0 & 0 & 1
    \end{bmatrix} p_0\rb)\\
    &= \sum_{a,b,c,d,e,f\in \FF_q}f\lb(\frac{1}{bd-ae}\begin{bmatrix}
    x_{11} & x_{12} & x_{13}\\
    x_{21} & x_{22} & x_{23}\\
    x_{31} & x_{32} & x_{33}
    \end{bmatrix} p_0\rb),
    \end{align*}
    where the entries of the matrix, by column, are as follows: \begin{itemize}[label=$\diamond$]
        \item $x_{11}=\alpha(bd-ae)+\beta(ab+cde-bdf)+\ep(ace-abf-de\delta)$
        \item $x_{21}=\beta(adf-a^2-cd^2)+\ep(a^2f-acd+d^2\delta)$
        \item $x_{31}=\beta d+\ep a$
        \item $x_{12}=\beta(ce^2-bef+b^2)+\ep(bce-b^2f-e^2\delta)$
        \item $x_{22}=\alpha(bd-ae)+\beta(aef-cde-ab)+\ep(abf-bcd+de\delta)$
        \item $x_{32}=\beta e+\ep b$
        \item $x_{13}=\beta (cef+bc-bf^2-e\delta)$
        \item $x_{23}=\beta(af^2-cdf-ac+d\delta)+\ep(acf-ac^2+df\delta-a\delta)$
        \item $x_{33}=\alpha+\beta f + \ep c$
    \end{itemize}
    Unfortunately, there does not seem to be an obvious way to simplify this expression further. When computing this case for a specific field $\FF_q$, it seems that applying this formula may not be more helpful than directly computing.
\end{itemize}

\subsection{Elliptic terms of the second kind} $\gamma=\begin{bmatrix}
  a & b & 0\\
  c & d & 0\\
  0 & 0 & e
\end{bmatrix}$, where $g=\begin{bmatrix}
  a & b\\
  c & d
\end{bmatrix}$ is irreducible, and $e\in \FF_p^{\times}$. Also, \[
\Gamma_{\gamma}=\lb\{\begin{bmatrix}
  a & b\xi & 0\\
  b & a & 0\\
  0 & 0 & c
\end{bmatrix}:\, c(a^2-b^2\xi)\neq 0\rb\},
\]where $\xi$ is a nonsquare in $\FF_p$.


Once again, depending on whether $n\equiv 0\pmod{2}$, i.e. whether $n$ is even or odd, there are two cases to consider.\begin{itemize}
    \item $n$ is even. In this case, $\FF_p\subset\FF_{p^2}\subseteq\FF_q$ and $\gamma$ is diagonalizable in $\mathrm{GL}_3(\FF_q)$. $G_{\gamma}$ is a subgroup of diagonal matrices over $\FF_q$, so the fundamental domain $G_\gamma\bk \HH_q$ is the same as for the second hyperbolic case.
    
    Suppose $\gamma$ is similar to $\Big(\begin{smallmatrix}\eta_1&0&0\\0&\eta_2&0\\0&0&e\\\end{smallmatrix}\Big)$ in $G$, where $\eta_1,\eta_2\in\FF_{p^2}\bk\FF_p$ and $e\in\FF_p^\times$. The orbital sum will be (analogously to the second hyperbolic case):
    
    $$I_G(f,\gamma)=|K/Z|\cdot Hf\lb(\begin{bmatrix}\eta_1/e&0\\0&\eta_2/e\end{bmatrix}\rb).$$
    
    As $\eta_1,\eta_2$ are Galois conjugates without loss of generality we may assume that $\eta_1=\eta_2^p$. Therefore the total contribution of the second elliptic terms will be:
    
    $$\frac12\cdot\frac{(q-1)^3}{(p-1)(p^2-1)}\cdot\frac{q^3-1}{q-1}\sum_{e\in\FF_p^\times}\sum_{\eta\in\FF_{p^2}^\times\bk\FF_p^\times}Hf\lb(\begin{bmatrix}
      \eta/e&0\\0&\eta^p/e
    \end{bmatrix}\rb)=$$
    
    $$=\frac12\cdot\frac{(q-1)^3}{(p^2-1)}\cdot\frac{q^3-1}{q-1}\sum_{\eta\in\FF_{p^2}^\times\bk\FF_p^\times}Hf\lb(\begin{bmatrix}
      \eta&0\\0&\eta^p
    \end{bmatrix}\rb).$$
    
    \item $n$ is odd. In this case $\gamma$ would not be diagonalizable in $\mathrm{GL}_3(\FF_q)$. In general, it will be conjugate to an element of the form:
    $$\begin{bmatrix}k&l\xi&0\\l&k&0\\0&0&m\end{bmatrix},$$where $k\in\FF_p$, $m,l,\xi\in\FF_p^\times$ such that $\xi$ is a nonsquare in $\FF_p$. (This ensures that the determinant $m(k^2-l^2\xi)$ is nonzero.) In this case we have:\[
    G_{\gamma} =\lb\{\begin{bmatrix}
  a & b\xi & 0\\
  b & a & 0\\
  0 & 0 & c
\end{bmatrix}:\,a,b,c\in\FF_q,c(a^2-b^2\xi)\neq 0\rb\}.
\]We seek a fundamental domain for this specific $G_{\gamma}$.

\begin{prop}
A fundamental domain for $G_\gamma$ is given by:
\[G_\gamma\bk\HH_q=\lb\{\lb(x+u\delta^{1/3}+v\delta^{2/3},y+\delta^{1/3}\rb):v\in\FF_q^{\times},x,y,u\in\FF_q\rb\}.\]
\end{prop}

The corresponding orbital sum is

$$I_G(f,\gamma)=\frac{(q^3-1)(q^2-1)}{(p-1)(p^2-1)}Hf\lb(\begin{bmatrix}
  k/m&l\xi/m\\l/m&k/m
\end{bmatrix}\rb).$$

The total contribution from the second elliptic terms in this case is

$$\frac{(q^3-1)(q^2-1)}{(p^2-1)}\sum_{\substack{s\in\FF_p\\t\in\FF_p^\times}}Hf\lb(\begin{bmatrix}
  s&t\xi\\t&s
\end{bmatrix}\rb).$$

A detailed proof of both the proposition and the orbital sum
calculation can be found in \hyperref[appA]{Appendix A}.

\end{itemize}

\section{Examples}

Throughout this section, subscripts under multiplicities indicate which group is the domain of the corresponding representations. Set $\rho=\Ind_\Gamma^G1$ and $\kappa=\Ind_K^G1$.

\begin{exmp}
Let $f:G\to\CC$ be the constant one function. The right-hand side of the trace formula will be $|G|$ while the left-hand side becomes (by orthogonality of characters):

$$\sum_{\pi\in\hat{G}}m(\pi,\rho)_G\langle f,\chi_\pi\rangle=\sum_{\pi\in\hat{G}}m(\pi,\rho)_G\langle \chi_{1},\chi_\pi\rangle=|G|\cdot m(1,\rho)_G.$$
So the trace formula says $|G|\cdot m(1,\rho)=|G|$ which is a manifestation of Frobenius reciprocity: $$m(1,\rho)_G=\frac{1}{|G|}\langle 1,\Ind_\Gamma^G1\rangle_G=\frac{1}{|\Gamma|}\langle \Res^G_\Gamma 1,1\rangle_\Gamma=m(1,1)_\Gamma=1.$$
\end{exmp}

\begin{exmp}
Let $f:\HH_q\to\CC$ be the indicator function $\delta_{p_0}$ which is  $1$ on $p_0$ and zero otherwise. It is $K-$bi-invariant because $\stab(p_0)=K$. On the left-hand side of the trace formula we have:

\begin{equation*}
    \begin{split}
        \sum_{\pi\in\hat{G}}m(\pi,\rho)_G\langle f,\chi_\pi\rangle&=\sum_{\pi\in\hat{G}}m(\pi,\rho)_G\sum_{g\in G}f(g\cdot p_0)\chi_\pi\lb(g^{-1}\rb)=
        \sum_{\pi\in\hat{G}}m(\pi,\rho)_G\sum_{k\in K}\chi_\pi\lb(k^{-1}\rb)=\\
        &=\sum_{\pi\in\hat{G}}m(\pi,\rho)_G\langle 1,\chi_\pi\rangle_K=|K|\cdot\sum_{\pi\in\hat{G}}m(\pi,\rho)_G\cdot m(1,\Res^G_K\pi)_K=\\
        &=|K|\cdot\sum_{\pi\in\hat{G}}m(\pi,\rho)_G\cdot m(\pi,\kappa)_G.
    \end{split}
\end{equation*}
On the right-hand side, the only nonzero terms are the central terms and the first elliptic terms when $n$ is not divisible by $3$. If that is the case, we claim that for a non-central $\gamma$ in $K$ we have $I_G(f,\gamma)=3.$ 

Suppose $\gamma$ corresponds to  multiplication by $\xi\in\FF_q\lb(\delta^{1/3}\rb)^\times$ under the isomorphism $K\simeq\FF_q\lb(\delta^{1/3}\rb)^\times$. When viewed as an element of $\GL_3(\FF_q)$, the eigenvalues of $\gamma$ are the Galois conjugates of $\xi$, which are $\xi$, $\xi^q$ and $\xi^{q^2}$ in our setting. This means that the conjugates of $\gamma$ lying in K are precisely $\gamma^q$ and $\gamma^{q^2}$, hence there will be exactly three non-zero terms in each orbital sum $I_G(f,\gamma)$.

The total contribution of the first elliptic terms will be:

$$\frac13\cdot\frac{q^3-1}{p^3-1}\sum_{\substack{\gamma\in K\\\gamma\not\in Z}}I_G(f,\gamma)=\frac{(q^3-1)(q^3-q)}{p^3-1}.$$

Thus we get:

\textit{Case 1:} $n$ is divisible by $3$. Then we have

$$\sum_{\pi\in\hat{G}}m(\pi,\rho)_G\cdot m(\pi,\kappa)_G=\frac{q^3(q^2-1)(q-1)}{p^3(p^2-1)(p^3-1)}.$$

\textit{Case 2:} $n$ is not divisible by $3$. Then we have 

$$\sum_{\pi\in\hat{G}}m(\pi,\rho)_G\cdot m(\pi,\kappa)_G=\frac{q^3(q^2-1)(q-1)}{p^3(p^2-1)(p^3-1)}+\frac{q^3-q}{p^3-1}.$$

\begin{rem}
    Note that the sum on the left-hand side can also be written as:
    
    $$\sum_{\pi\in\hat{G}}m(\pi,\rho)_G\cdot m(\pi,\kappa)_G=\frac{1}{|G|}\langle\chi_\rho,\chi_\kappa\rangle.$$
\end{rem}

\end{exmp}
\section{Character of the representation \texorpdfstring{$\rho = \Ind_{\Gamma}^{G} 1$}{}}

Having effectively computed the right-hand side of the pre-trace formula \eqref{eq:4}, we now turn to the left-hand side. For this, we require the character $\chi_{\rho}$ of $\rho$, which we now calculate. We use the usual formula for the character of an induced representation (see \cite[Ch.\ 16]{terras_1999}):
\[\chi_{\rho}(\gamma) = \frac{1}{|\Gamma|}\sum_{x\in G} \mathbbm{1}_{\Gamma}(x\gamma x^{-1}).\]

We treat each conjugacy class. By $H_{1,q},H_{2,q},P_{1,q},\dots,E_{2,q}$ we denote the size of various types of conjugacy classes in $\GL_3(\FF_q)$. Given a set $S$, we use $\mathbbm{1}_S$ to denote the indicator function for $S$.

\begin{itemize}
\item Central class: $\gamma = \begin{bmatrix}
    a & 0 & 0\\
    0 & a & 0\\
    0 & 0 & a
  \end{bmatrix}, a \in \FF_{q}^{\times}.$ Since $\gamma\in Z(G)$, we have $x\gamma x^{-1} = \gamma$. So, 
  \[\chi_{\rho}(\gamma) = \frac{|G|}{|\Gamma|}\mathbbm{1}_{\FF_p}(a).\]
  \item First hyperbolic class: $\gamma =
  \begin{bmatrix}
    a & 0 & 0\\
    0 & a & 0\\
    0 & 0 & b
  \end{bmatrix}$ with $a\neq b \in \FF_{q}^{\times}.$ Now, $\gamma$ is similar to an element of $\Gamma$ only when $a,b$ are roots of a polynomial over $\FF_p$. This occurs only when $a,b\in \FF_p^{\times}$; for instance, this can be seen by considering the determinant $\Det(\gamma) = a^2b$. Being the constant coefficient of $\text{char}_\gamma \in \FF_{p}[x]$, we have $\Det(\gamma)\in \FF_p$, which implies $\Det(\gamma)$ is invariant under the Frobenius automorphism (as $\Gal(\FF_q/\FF_p)$ is generated by the Frobenius). Since $\FF_q/\FF_p$ is Galois, if $a,b\not\in \FF_{p}$, they must be Galois conjugates, i.e., $a^p = b$ and $b^p = a$. But this forces $a^2b = (a^2b)^p = b^2a,$ i.e., $a=b$, a contradiction. So, by a simple application of the orbit-stabilizer theorem, it follows that
  \[\chi_{\rho}(\gamma) = \frac{|G|}{|\Gamma|} \frac{H_{1,p}}{H_{1,q}}\cset_{\Gamma}(\gamma).\]

  \item Second hyperbolic class: $\gamma =
  \begin{bmatrix}
    a & 0 & 0\\
    0 & b & 0\\
    0 & 0 & c
  \end{bmatrix}$ with $a,b,c\in \FF_q^{\times}$ distinct. There are three cases which allow $\gamma$ to be similar to an element of $\Gamma$: either all three are Galois conjugates in $\FF_{p^{3}}/\FF_p$ (only possible if $3\mid n$), or $a,b$ are Galois conjugates in $\FF_{p^2}/\FF_p$ and $c\in \FF_p$ (only possible if $2\mid n$), or all three belong to $\FF_p$. These cases give rise to the three terms in the sum:
  \[\chi_{\rho}(\gamma) = \frac{|G|}{|\Gamma|\cdot H_{2,q}}\lb(E_{2,p}\cdot \cset_{\FF_{p^3}\setminus\FF_p}(a)\delta_{b}(a^p)\delta_{c}(a^{p^2}) +E_{1,p}\cdot \cset_{\FF_{p^2}\setminus\FF_p}(a)\delta_b(a^p)\cset_{\FF_p}(c)+ H_{2,p}\cdot \cset_{\Gamma}(\gamma)\rb).\]
  
 \item First parabolic class: $\gamma =
  \begin{bmatrix}
    a & 1 & 0\\
    0 & a & 0\\
    0 & 0 & a
  \end{bmatrix}, a\in \FF_{q}^{\times}$. If $a\not \in \FF_p$ was the root of a polynomial over $\FF_p$, its Galois conjugate $a^p \neq a$ would also be a root, which is clearly not the case. So, 
  \[\chi_{\rho}(\gamma) = \frac{|G|}{|\Gamma|}\frac{P_{1,p}}{P_{1,q}}\cset_{\FF_p}(a). \]
  \item Second parabolic class: $\gamma =
  \begin{bmatrix}
    a & 1 & 0\\
    0 & a & 1\\
    0 & 0 & a
  \end{bmatrix}, a \in \FF_{q}^{\times}.$ Similar to the previous class,
  \[\chi_{\rho}(\gamma) = \frac{|G|}{|\Gamma|}\frac{P_{2,p}}{P_{2,q}}\cset_{\FF_p}(a). \]
  \item Third parabolic class: $\gamma =
  \begin{bmatrix}
    a & 1 & 0\\
    0 & a & 0\\
    0 & 0 & b
  \end{bmatrix}$ with $a\neq b \in \FF_{q}^{\times}$. By a similar argument as for the first hyperbolic class, we have 
  \[\chi_{\rho}(g) = \frac{|G|}{|\Gamma|} \frac{P_{3,p}}{P_{3,q}}\cset_{\Gamma}(g).\]

  \item First elliptic class: $\gamma =
  \begin{bmatrix}
    w & 0 & 0\\
    0 & w^{q} & 0\\
    0 & 0 & r
  \end{bmatrix}$ with $w \in \FF_{q^2}\setminus\FF_q$ and $r\in \FF_q^{\times}$. Here $\gamma$ is similar to an element of $\Gamma$ if and only if $w,w^{q}$ are roots of a degree-2 irreducible polynomial over $\FF_p$. However, if $2\mid n$, then $\FF_{p^2}$ is an intermediate extension in $\FF_q/\FF_p$ and so roots of any degree-2 polynomial over $\FF_p$ are included in $\FF_{q}$. Thus, $\chi_{\rho}(\gamma) = 0$ if $n$ is divisible by 2, otherwise
 \[\chi_{\rho}(\gamma) = \frac{|G|}{|\Gamma|}\frac{E_{1,p}}{E_{1,q}}\cset_{\FF_{p^2}\setminus\FF_p}(w)\cset_{\FF_p}(r).\]
  \item Second elliptic class: $\gamma =
  \begin{bmatrix}
    w & 0 & 0\\
    0 & w^{q} & 0\\
    0 & 0 & w^{q^2}
  \end{bmatrix}, w \in \FF_q(\delta^{1/3})\setminus\FF_q$. As for the previous class, $\chi_{\rho}(\gamma) = 0$ if $n$ is divisible by 3, otherwise
  \[\chi_{\rho}(\gamma) = \frac{|G|}{|\Gamma|}\frac{E_{2,p}}{E_{2,q}}\cset_{\FF_{p^3}\setminus\FF_p}(w).\]
  \end{itemize}
  
    \section{Decomposition of $\rho$ when $2,3\nmid n$}
  
    In this section, we compute the decomposition of $\rho$ when neither $2$ nor $3$ divide $n$; thus elliptic elements in $\GL_3(\FF_p)$ remain elliptic in $\GL_3(\FF_q)$.

    We use the character table for $\GL_3(\FF_q)$ in Terras \cite[pp. 383-4]{terras_1999} and Steinberg \cite{steinberg_1951}. As usual, to calculate multiplicity $m(\pi, \rho)$ of an irreducible representation $\pi$ present in $\rho$, we compute the inner product of their respective characters. By $H_1,H_2,P_1\dots,E_2$ we denote the size of different types of conjugacy classes in $\GL_3(\FF_p)$. By $\Nm_{E/F}$, we denote the norm of the field extension $E/F$.
  
  \begin{enumerate}
      \item $\alpha$, a character of $\FF_q^{\times}$. Then, 
      \begin{align*}
          m(\alpha, \rho) &= \frac{1}{|\Gamma|}\Bigg(\sum_{a\in \FF_p^{\times}} \alpha(a^3) + H_1\sum_{\substack{a,b\in \FF_{p}^{\times}\\a\neq b}}\alpha(a^2b) + \frac{H_2}{6}\sum_{\substack{a,b,c\in \FF_p^{\times}\\a,b,c \text{ distinct}}} \alpha(abc) + P_1\sum_{a\in \FF_p^{\times}} \alpha(a^3)+P_2\sum_{a\in \FF_p^{\times}} \alpha(a^3) \\
          &\qquad\qquad + P_3\sum_{\substack{a,b\in \FF_{p}^{\times}\\a\neq b}}\alpha(a^2b)+ \frac{E_1}{2}\sum_{\substack{w\in \FF_{p^2}\setminus\FF_{p}\\r\in \FF_p^{\times}}}\alpha(r \Nm_{\FF_{q^2}/\FF_q}(w)) +\frac{E_2}{3}\sum_{\substack{w\in \FF_{p^3}\setminus\FF_{p}}}\alpha( \Nm_{\FF_{q^3}/\FF_q} (w))\Bigg). 
      \end{align*}
      We can simplify these sums on the basis of three cases:
      
      \begin{itemize}
          \item $\alpha^3\mid_{\FF_p^{\times}} \neq 1$.  Each of the sums is zero and so $m(\alpha, \rho) = 0$.
          \item $\alpha^3\mid_{\FF_p^{\times}} = 1$ but $\alpha\mid_{\FF_p^{\times}} \neq 1$. A short calculation yields
          \[m(\alpha,\rho) = \frac{1}{|\Gamma|}\lb(1-H_1+\frac{H_2}{3} + P_1+P_2-P_3-\frac{E_2}{3}\rb)(p-1) = 0.\]
          \item $\alpha\mid_{\FF_p^{\times}} = 1$. Then simply
      \[m(\alpha,\rho) = \frac{1}{|\Gamma|}\sum_{\{g\}\in C_{\Gamma}} |\{g\}| = 1.\] 
      \end{itemize}
      
      \item $\pi_{\alpha}$, where $\alpha$ is a character of $\FF_q^{\times}$.
        \begin{align*}
          m(\pi_{\alpha}, \rho) &= \frac{1}{|\Gamma|}\Bigg((q^2+q)\sum_{a\in \FF_p^{\times}} \alpha(a^3) + (q+1)H_1\sum_{\substack{a,b\in \FF_{p}^{\times}\\a\neq b}}\alpha(a^2b) + \frac{H_2}{3}\sum_{\substack{a,b,c\in \FF_p^{\times}\\a,b,c \text{ distinct}}} \alpha(abc) + qP_1\sum_{a\in \FF_p^{\times}} \alpha(a^3)\\
          &\qquad\qquad +0+ P_3\sum_{\substack{a,b\in \FF_{p}^{\times}\\a\neq b}}\alpha(a^2b)+ 0 - \frac{E_2}{3}\sum_{\substack{w\in \FF_{p^3}\setminus\FF_{p}}}\alpha( \Nm_{\FF_{q^3}/\FF_q}(w))\Bigg). 
      \end{align*}
      
      \begin{itemize}
          \item $\alpha^3\mid_{\FF_p^{\times}} \neq 1$. Here again all the sums are zero, giving $m(\pi_{\alpha},\rho) = 0$.
          \item $\alpha^3\mid_{\FF_p^{\times}} = 1$ but $\alpha\mid_{\FF_p^{\times}} \neq 1$. 
          \begin{align*}
           m(\pi_{\alpha},\rho) &= \frac{1}{|\Gamma|}\lb((q^2+q)-(q+1)H_1 + \frac{2}{3}H_2 +qP_1-P_3 + \frac{E_2}{3} \rb)(p-1)\\
           &= \frac{(q-p)(q-p^2)}{p^3(p-1)^2(p+1)(p^2+p+1)}.
       \end{align*}
          \item $\alpha\mid_{\FF_p^{\times}} = 1$.  \begin{align*}
          m(\pi_{\alpha}, \rho) &= \frac{1}{|\Gamma|}\Bigg( (q^2+q)(p-1)+(q+1)H_1(p-1)(p-2)+(p-1)(p-2)(p-3) \frac{H_2}{3} \\
          & \qquad\qquad+ q P_1(p-1) + (p-1)(p-2) P_3 -  \frac{E_2}{3} (p^3 - p)\Bigg)\\ 
          & = \frac{(q-p)(q + p^5 -2 p^2)}{p^3(p-1)^2(p+1)(p^2+p+1)}.
      \end{align*}
      As a sanity check, showing this is actually an integer is not very hard.
      \end{itemize}

      \item $\pi_{\alpha}',$ where $\alpha$ is a character of $\FF_q^{\times}$.
       \begin{align*}
          m(\pi_{\alpha}', \rho) &= \frac{1}{|\Gamma|}\Bigg(q^3\sum_{a\in \FF_p^{\times}} \alpha(a^3) + qH_1\sum_{\substack{a,b\in \FF_{p}^{\times}\\a\neq b}}\alpha(a^2b) + \frac{H_2}{6}\sum_{\substack{a,b,c\in \FF_p^{\times}\\a,b,c \text{ distinct}}} \alpha(abc) + 0+ 0+ 0\\
          &\qquad\qquad  - \frac{E_1}{2}\sum_{\substack{w\in \FF_{p^2}\setminus\FF_{p}\\r\in \FF_p^{\times}}}\alpha(r \Nm_{\FF_{q^2}/\FF_q}(w)) + \frac{E_2}{3}\sum_{\substack{w\in \FF_{p^3}\setminus\FF_{p}}}\alpha(\Nm_{\FF_{q^3}/\FF_q}(w))\Bigg). 
      \end{align*}
      
          \begin{itemize}
          \item $\alpha^3\mid_{\FF_p^{\times}} \neq 1$. $m(\pi_{\alpha},\rho) = 0$.
          \item $\alpha^3\mid_{\FF_p^{\times}} = 1$ but $\alpha\mid_{\FF_p^{\times}} \neq 1$. 
          \begin{align*}
         m(\pi_{\alpha}',\rho) = \frac{1}{|\Gamma|}\lb(q^3- qH_1 + \frac{1}{3}H_2 - \frac{E_2}{3} \rb)(p-1) = \frac{(q-p)(q-p^2)(q+p^2+p)}{p^3(p-1)^2(p+1)(p^2+p+1)}.
       \end{align*}
          \item $\alpha\mid_{\FF_p^{\times}} = 1$.   \begin{align*}
          m(\pi_{\alpha}', \rho) &=  \frac{1}{|\Gamma|}\Bigg(q^3(p-1)+q(p-1)(p-2)H_1+(p-1)(p-2)(p-3)\frac{H_2}{6}\\ 
          & \qquad\qquad -(p^2-p)(p-1)\frac{E_1}{2} +(p^3-p)\frac{E_2}{3}\Bigg)\\
          & = \frac{(q - p)(q^2+pq+p^5-p^4-p^3-p^2)}{p^3(p-1)^2(p+1)(p^2+p+1)}.
      \end{align*}
      \end{itemize}

       \item $\pi_{\alpha,\beta}$ for distinct characters $\alpha$ and $\beta$ of $\FF_q^{\times}$.
      \begin{align*}
          m(\pi_{\alpha,\beta}, \rho) &= \frac{1}{|\Gamma|}\Bigg((q^2+q+1)\sum_{a\in \FF_p^{\times}} \alpha(a^2)\beta(a) + H_1\sum_{\substack{a,b\in \FF_{p}^{\times}\\a\neq b}}((q+1)\alpha(ab)\beta(a)+\alpha(a^2)\beta(b))\\
         & \qquad\qquad+ \frac{H_2}{6}\sum_{\substack{a,b,c\in \FF_p^{\times}\\a,b,c \text{ distinct}}} (\beta(a)\alpha(bc)+\beta(b)\alpha(ac)+\beta(c)\alpha(ab)) + (q+1)P_1\sum_{a\in \FF_p^{\times}} \alpha(a^2)\beta(a)\\
                                      &\qquad\qquad+P_2\sum_{a\in \FF_p^{\times}} \alpha(a^2)\beta(a) + P_3\sum_{\substack{a,b\in \FF_{p}^{\times}\\a\neq b}}(\alpha(ab)\beta(a)+\alpha(a^2)\beta(b))\\
        &\qquad\qquad+ \frac{E_1}{2}\sum_{\substack{w\in \FF_{p^2}\setminus\FF_{p}\\r\in \FF_p^{\times}}}\alpha(\Nm_{\FF_{q^2}/\FF_q}(w))\beta(r) +0\Bigg).
      \end{align*}
      \begin{itemize}
          \item $\alpha^2 \beta \mid_{\FF_p^{\times}}\neq 1$. $m(\pi_{\alpha,\beta},\rho) = 0$.
          \item $\alpha^2 \beta \mid_{\Fmp}=1$ but $\alpha^2\mid_{\Fmp}\neq 1$ and $\beta\mid_{\Fmp}\neq 1$. 
          \begin{align*}
         m(\pi_{\alpha,\beta},\rho) &= \frac{1}{|\Gamma|}\Bigg((q^2+q+1)(p-1) - (q+2)(p-1)H_1 + 
      (p-1)H_2 \\
      & \qquad\qquad +(q+1)(p-1)P_1 +(p-1)P_2 -2(p-1)P_3\Bigg) \\
      & = \frac{(q-p)(q-p^2)}{p^3(p-1)^2(p+1)(p^2+p+1)}.
     \end{align*}
     \item $\alpha^2\mid_{\Fmp} = 1$ and $\beta \mid_{\Fmp}=1$ but $\alpha\mid_{\Fmp}\neq 1$.
      \begin{align*}
         m(\pi_{\alpha,\beta},\rho) &= \frac{1}{|\Gamma|}\Bigg((q^2+q+1)(p-1) -(p-1)(q-p+3)H_1-(p-1)(p-3)\frac{H_2}{2}\\
         &\qquad\qquad +(q+1)(p-1)P_1 +(p-1)P_2 + (p-1)(p-3)P_3 - (p-1)^2\frac{E_1}{2}\Bigg) \\
      & = \frac{(q-p)(q-p^2)}{p^3(p-1)^2(p+1)(p^2+p+1)}.
     \end{align*}
    \item $\alpha\mid_{\Fmp} = 1$ and $\beta \mid_{\Fmp}=1$.
    \begin{align*}
     m(\pi_{\alpha,\beta},\rho) &= \frac{1}{|\Gamma|}\Bigg((q^2+q+1)(p-1) + (q+2)(p-1)(p-2)H_1+(p-1)(p-2)(p-3)\frac{H_2}{2}\\
         &\qquad\qquad +(q+1)(p-1)P_1 +(p-1)P_2 + 2(p-1)(p-2)P_3 + (p-1)(p^2-p)\frac{E_1}{2}\Bigg) \\
      &  = \frac{q(q+p^5-2p^2-p) + p^3(p^5-2p^3-p^2+3)}{p^3(p-1)^2(p+1)(p^2+p+1)}.
    \end{align*}
      \end{itemize}
      
      \item $\pi_{\alpha,\beta}'$ for distinct characters $\alpha$ and $\beta$ of $\FF_q^{\times}$. 
     \begin{align*}
          m(\pi_{\alpha,\beta}', \rho) &= \frac{1}{|\Gamma|}\Bigg(q(q^2+q+1)\sum_{a\in \FF_p^{\times}} \alpha(a^2)\beta(a) + H_1\sum_{\substack{a,b\in \FF_{p}^{\times}\\a\neq b}}((q+1)\alpha(ab)\beta(a)+q\alpha(a^2)\beta(b))\\
         & \qquad\qquad+ \frac{H_2}{6}\sum_{\substack{a,b,c\in \FF_p^{\times}\\a,b,c \text{ distinct}}} (\beta(a)\alpha(bc)+\beta(b)\alpha(ac)+\beta(c)\alpha(ab)) + qP_1\sum_{a\in \FF_p^{\times}} \alpha(a^2)\beta(a)\\
         &\qquad\qquad+0 + P_3\sum_{\substack{a,b\in \FF_{p}^{\times}\\a\neq b}}\alpha(ab)\beta(a) - \frac{E_1}{2}\sum_{\substack{w\in \FF_{p^2}\setminus\FF_{p}\\r\in \FF_p^{\times}}}\alpha(\Nm_{\FF_{q^2}/\FF_q}(w))\beta(r) +0\Bigg).
      \end{align*}
     \begin{itemize}
          \item $\alpha^2 \beta \mid_{\FF_p^{\times}}\neq 1$. $m(\pi_{\alpha,\beta}',\rho) = 0$.
          
          \item $\alpha^2 \beta \mid_{\Fmp}=1$ but $\alpha^2\mid_{\Fmp}\neq 1$ and $\beta\mid_{\Fmp}\neq 1$. 
          \begin{align*}
         m(\pi_{\alpha,\beta}',\rho) &= \frac{1}{|\Gamma|}\Bigg(q(q^2+q+1)(p-1) - (2q+1)(p-1)H_1 + 
      (p-1)H_2 \\
      & \qquad\qquad +q(p-1)P_1  -(p-1)P_3\Bigg) \\
      & = \frac{(q - p) (q-p^2 ) (q+ p^2 + p + 1)}{p^3(p-1)^2(p+1)(p^2+p+1)}.
     \end{align*}
     
     \item $\alpha^2\mid_{\Fmp} = 1$ and $\beta \mid_{\Fmp}=1$ but $\alpha\mid_{\Fmp}\neq 1$.
      \begin{align*}
         m(\pi_{\alpha,\beta},\rho) &= \frac{1}{|\Gamma|}\Bigg(q(q^2+q+1)(p-1) +(p-1)(qp-3q-1)H_1-(p-1)(p-3)\frac{H_2}{2}\\
         &\qquad\qquad +q(p-1)P_1 - (p-1)P_3 + (p-1)^2\frac{E_1}{2}\Bigg) \\
      &  = \frac{(q-p)(q(q+p+1)+p^2(p^3-p^2-p-2))}{p^3(p-1)^2(p+1)(p^2+p+1)}.     \end{align*}
     
    \item $\alpha\mid_{\Fmp} = 1$ and $\beta \mid_{\Fmp}=1$.
    \begin{align*}
     m(\pi_{\alpha,\beta},\rho) &= \frac{1}{|\Gamma|}\Bigg(q(q^2+q+1)(p-1) + (2q+1)(p-1)(p-2)H_1+(p-1)(p-2)(p-3)\frac{H_2}{2}\\
         &\qquad\qquad +q(p-1)P_1+ (p-1)(p-2)P_3 - (p-1)(p^2-p)\frac{E_1}{2}\Bigg) \\
      & = \frac{(q - p) (q(q + p + 1) + p^2 (2 p-3) ( p^2+p+1))}{p^3(p-1)^2(p+1)(p^2+p+1)}.
    \end{align*}

      \end{itemize}
      
   \item $\pi_{\alpha,\beta,\gamma}$ for distinct characters $\alpha,\beta,\gamma$ of $\FF_{q}^{\times}$.

         \begin{align*}
          m(\pi_{\alpha,\beta,\gamma}, \rho) &= \frac{1}{|\Gamma|}\Bigg((q+1)(q^2+q+1)\sum_{a\in \FF_p^{\times}} \alpha(a)\beta(a)\gamma(a)+ H_2\sum_{\substack{a,b,c\in \FF_p^{\times}\\a,b,c \text{ distinct}}} \alpha(a)\beta(b)\gamma(c) \\
         & \qquad\qquad + (q+1)H_1\sum_{\substack{a,b\in \FF_{p}^{\times}\\a\neq b}}(\alpha(b)\beta(a)\gamma(a)+\beta(b)\gamma(a)\alpha(a)+\gamma(b)\alpha(a)\beta(a)) \\
         &\qquad\qquad + (2q+1)P_1\sum_{a\in \FF_p^{\times}} \alpha(a)\beta(a)\gamma(a)+ P_2\sum_{a\in \FF_p^{\times}} \alpha(a)\beta(a)\gamma(a)\\
         &\qquad\qquad  + P_3\sum_{\substack{a,b\in \FF_{p}^{\times}\\a\neq b}}(\alpha(b)\beta(a)\gamma(a)+\beta(b)\gamma(a)\alpha(a)+\gamma(b)\alpha(a)\beta(a))+0+0\Bigg).
      \end{align*}
      \begin{itemize}
       \item $\alpha \beta \gamma \mid_{\FF_p^{\times}}\neq 1$. $m(\pi_{\alpha,\beta,\gamma},\rho) = 0$.
          
        \item $\alpha\beta \gamma \mid_{\Fmp}  = 1$ but none of $\alpha,\beta,$ or $\gamma$ restrict to the trivial character on $\Fmp.$
        \begin{align*}
         m(\pi_{\alpha,\beta,\gamma},\rho) & = \frac{1}{|\Gamma|}\Bigg((q+1)(q^2+q+1)(p-1) + 2(p-1)H_2 - 3(q+1)(p-1)H_1 \\
         &\qquad\qquad+ (2q+1)(p-1)P_1 + (p-1)P_2-3(p-1)P_3\Bigg)\\
         & = \frac{(q - p) (q-p^2) (q+p^2+p+2)}{p^3(p-1)^2(p+1)(p^2+p+1)}.
        \end{align*}
        \item $\alpha\beta\mid_{\Fmp}  = 1$ and $\gamma\mid_{\Fmp} = 1$ but $\alpha$ and $\beta$ are nontrivial on $\Fmp$. Note that since $m(\pi_{\alpha,\beta,\gamma},\rho)$ is symmetric in $\alpha,\beta,\gamma$, the calculation here also holds for the other two permutations of this case.
         \begin{align*}
         m(\pi_{\alpha,\beta,\gamma},\rho) & = \frac{1}{|\Gamma|}\Bigg((q+1)(q^2+q+1)(p-1) -(p-1)(p-3)H_2 + (q+1)(p-1)(p-4)H_1 \\
         &\qquad\qquad+ (2q+1)(p-1)P_1 + (p-1)P_2+(p-1)(p-4)P_3\Bigg)\\
         & = \frac{(q-p)(q(q+p+2)+p^2(p^3-p^2-p-3))}{p^3(p-1)^2(p+1)(p^2+p+1)}.
        \end{align*}
        \item All three $\alpha,\beta,\gamma$ restrict to the trivial character on $\Fmp$. 
         \begin{align*}
         m(\pi_{\alpha,\beta,\gamma},\rho) & = \frac{1}{|\Gamma|}\Bigg((q+1)(q^2+q+1)(p-1) +(p-1)(p-2)(p-3)H_2 + 3(q+1)(p-1)(p-2)H_1 \\
         &\qquad\qquad+ (2q+1)(p-1)P_1 + (p-1)P_2+3(p-1)(p-2)P_3\Bigg)\\
         & = \frac{(q(q^2+2q+3p^5-p^4-p^3-6p^2-2p)+p^3(p^5-4p^3+p+6))}{p^3(p-1)^2(p+1)(p^2+p+1)}.
        \end{align*}
      \end{itemize}

      \item $\rho_{\alpha,\nu}$ for $\alpha$ a character of $\FF_q^{\times}$ and $\nu$ a character of $\FF_{q^2}^{\times}$ such that $\nu^{q} \neq \nu$ (i.e., it is nondecomposable).  
        \begin{align*}
          m(\rho_{\alpha,\nu}, \rho) &= \frac{1}{|\Gamma|}\Bigg((q-1)(q^2+q+1)\sum_{a\in \FF_p^{\times}} \alpha(a)\nu(a)+ (q-1)H_1\sum_{\substack{a,b\in \FF_{p}^{\times}\\a\neq b}}\alpha(b)\nu(a) + 0 \\
         &\qquad\qquad-P_1\sum_{a\in \FF_p^{\times}} \alpha(a)\nu(a)- P_2\sum_{a\in \FF_p^{\times}} \alpha(a)\nu(a)- P_3\sum_{\substack{a,b\in \FF_{p}^{\times}\\a\neq b}}\alpha(b)\nu(a)\\
         &\qquad\qquad  -\frac{E_1}{2}\sum_{\substack{w\in \FF_{p^2}\setminus \FF_p\\r\in \Fmp}}\alpha(r)(\nu(w)+\nu(w^p))+0\Bigg).
      \end{align*}
      Note that the last nonzero sum can be written as $-E_1\sum_{\substack{w\in \FF_{p^2}\setminus \FF_p\\r\in \Fmp}}\alpha(r)\nu(w).$ 
      \begin{itemize}
          \item $\alpha\nu \mid_{\FF_p^{\times}}\neq 1$. $m(\rho_{\alpha,\nu},\rho) = 0$.
          \item $\alpha \nu \mid_{\Fmp}=1$ but $\alpha\mid_{\Fmp}\neq 1$ and $\nu\mid_{\Fmp}\neq 1$. 
          \begin{align*}
         m(\rho_{\alpha,\nu},\rho) &= \frac{1}{|\Gamma|}\Bigg((q-1)(q^2+q+1)(p-1) - (q-1)(p-1)H_1\\
      & \qquad\qquad -(p-1)P_1 -(p-1)P_2 +(p-1)P_3\Bigg) \\
      & = \frac{(q - p) (q-p^2) (q + p^2 + p)}{p^3(p-1)^2(p+1)(p^2+p+1)}.
     \end{align*}
    \item $\alpha\mid_{\Fmp} = 1$, $\nu\mid_{\FF_{p^{3}}^{\times}}\neq 1$ but $\nu \mid_{\Fmp}=1$.
    \begin{align*}
     m(\rho_{\alpha,\nu},\rho) &= \frac{1}{|\Gamma|}\Bigg((q-1)(q^2+q+1)(p-1) + (q-1)(p-1)(p-2)H_1\\
         &\qquad\qquad -(p-1)P_1 -(p-1)P_2 - (p-1)(p-2)P_3 + (p-1)^2E_1\Bigg) \\
      & = \frac{(q-p)(q(q+p)+ p^2(p^3-p^2-p-1))}{(p-1)^2p^3(p+1)(p^2+p+1)}.
    \end{align*}
    
    \item $\alpha\mid_{\Fmp} = 1$ and $\nu\mid_{\FF_{p^{3}}^{\times}}= 1$.
    
\begin{align*}
     m(\rho_{\alpha,\nu},\rho) &= \frac{1}{|\Gamma|}\Bigg((q-1)(q^2+q+1)(p-1) + (q-1)(p-1)(p-2)H_1\\
         &\qquad\qquad -(p-1)P_1 -(p-1)P_2 - (p-1)(p-2)P_3 - (p-1)(p^2-p)E_1\Bigg) \\
      & = \frac{q(q^2+p^5-p^4-p^3-2p^2)+p^4(-p^4+2p+1)}{p^3(p-1)^2(p+1)(p^2+p+1)} .
    \end{align*}
    
      \end{itemize}
      
      \item $\sigma_{\mu}$ for $\mu$ a character of $\FF_{q^3}^{\times}$ such that $\mu^q \neq \mu$. 
      
       \begin{align*}
          m(\sigma_{\mu}, \rho) &= \frac{1}{|\Gamma|}\Bigg((q-1)^2(q+1)\sum_{a\in \FF_p^{\times}} \mu(a)+ 0 + 0-(q-1)P_1\sum_{a\in \FF_p^{\times}}\mu(a)+ P_2\sum_{a\in \FF_p^{\times}}\mu(a)+0\\
         &\qquad\qquad  +0+\frac{E_2}{3}\sum_{w\in \FF_{p^3}\setminus \FF_p}(\mu(w)+\mu(w^p)+\mu(w^{p^2}))\Bigg).
      \end{align*}
      \begin{itemize}
          \item $\mu\mid_{\Fmp} \neq 1$. $m(\sigma_{\mu},\rho) = 0$. 
          \item $\mu\mid_{\FF_{p^3}^{\times}} \neq 1$ but $\mu\mid_{\Fmp} = 1$. 
       \begin{align*}
          m(\sigma_{\mu},\rho) &= \frac{1}{|\Gamma|}\Bigg((q-1)^2(q+1)(p-1)-(q-1)(p-1)P_1+(p-1)P_2 - (p-1)E_2\Bigg)\\
          & = \frac{(q-p)(q-p^2)(q+p^2+p-1)}{p^3(p-1)^2(p+1)(p^2+p+1)}.
      \end{align*}
       \item $\mu\mid_{\FF_{p^3}^{\times}} = 1$.
       \begin{align*}
          m(\sigma_{\mu},\rho) &= \frac{1}{|\Gamma|}\Bigg((q-1)^2(q+1)(p-1)-(q-1)(p-1)P_1+(p-1)P_2 + (p^3-p)E_2\Bigg)\\
          & = \frac{q(q^2-q-p^4-p^3+p) + p^4(p^4-p^2+1)}{p^3(p-1)^2(p+1)(p^2+p+1)} .
      \end{align*}
      \end{itemize}
           
 \end{enumerate}


\newpage
\section{Appendix A: Fundamental Domains and Orbital Sums}\label{appA}
In this appendix, we give proofs of fundamental domains as well as some computations of orbital sums.

\begin{itemize}

\item \textbf{Proof of Lemma 2.2}
  \begin{proof}
  Observe that if $\alpha,\beta\in\FF_q\lb(\delta^{1/3}\rb)$ then $(\alpha,\beta)\in\HH_q$ if and only if $\{\alpha,\beta,1\}$ form a basis of $\FF_q\lb(\delta^{1/3}\rb)$ over $\FF_q$.
 
 Suppose $(\alpha,\beta)\in\HH_q$ and $x,y,z\in\FF_q$ are such that:
 
 $$x\cdot\frac{a\alpha+b\beta+c}{r\alpha+s\beta+t}+y\cdot\frac{d\alpha+e\beta+f}{r\alpha+s\beta+t}+z\cdot 1=0.$$
 
 The denominator $r\alpha+s\beta+t$ could only be nonzero if $r=s=t=0$ (as $\alpha,\beta,1$ are linearly independent) but that cannot happen as $\begin{vmatrix}
    a & b & c\\
    d & e & f\\
    r & s & t
  \end{vmatrix}\neq 0$. Thus we find:
  
  $$x(a\alpha+b\beta+c)+y(d\alpha+e\beta+f)+z(r\alpha+s\beta+t)=0$$
  
  In matrix form:
  $$\begin{bmatrix}x & y & z\end{bmatrix}\begin{bmatrix}
    a & b & c\\
    d & e & f\\
    r & s & t
  \end{bmatrix}\begin{bmatrix}\alpha\\ \beta\\ 1\end{bmatrix}=0.$$
This implies that $\begin{bmatrix}x & y & z\end{bmatrix}\begin{bmatrix}
    a & b & c\\
    d & e & f\\
    r & s & t
  \end{bmatrix}=\begin{bmatrix}0&0&0\end{bmatrix}$ because $\alpha,\beta,1$ are linearly independent over $\FF_q$. Since our matrix is nonsingular, it follows that $x=y=z=0$ and $\left\{\frac{a\alpha+b\beta+c}{r\alpha+s\beta+t},\frac{d\alpha+e\beta+f}{r\alpha+s\beta+t},1\right\}$ is a basis of $\FF_q\lb(\delta^{1/3}\rb)$ over $\FF_q$, proving the claim that $\GL_3(\FF_q)$ has an action on $\HH_q$ of the above form.
\end{proof}

    \item \textbf{Second Hyperbolic Term.}
 \begin{align*}
    I_G(f,\gamma) &= \sum_{t\in G_{\gamma}\bk G/K} \sum_{u\in K/\{aI\}}  f((tu)^{-1}\gamma (tu))\\
                  & = \sum_{t\in G_{\gamma}\bk \HH_q} \frac{|K|}{|\{aI\}|} f(t^{-1}\gamma t)\\
    & = \frac{q^3-1}{q-1}\sum_{\substack{x,y,r,s\in \FF_q\\ys\neq 1}} f\lb(
\begin{bmatrix}
 \frac{s}{s y-1} & \frac{1}{1-s y} & \frac{r-s x}{s y-1} \\
 \frac{1}{1-s y} & \frac{y}{s y-1} & \frac{x-r y}{s y-1} \\
 0 & 0 & 1 \\
\end{bmatrix}
\gamma
\begin{bmatrix}
 y & 1 & x \\
 1 & s & r \\
 0 & 0 & 1 \\
\end{bmatrix} p_0
                \rb)+ \\
    & +\frac{q^3-1}{q-1}\sum_{\substack{x,r,s\in\FF_q \\ s\neq 0}}f\lb(\begin{bmatrix}1&0&-x\\-s&1&xs-r\\0&0&1\end{bmatrix}\gamma\begin{bmatrix}1&0&x\\s&1&r\\0&0&1\\ \end{bmatrix}p_0\rb)+\\
    & + \frac{q^3-1}{q-1}\sum_{\substack{x,y,r\in\FF_q \\ y\neq 0}}f\lb(\begin{bmatrix}1&-y&-x+yr\\0&1&-r\\0&0&1\end{bmatrix}\gamma\begin{bmatrix}1&y&x\\0&1&r\\0&0&1\end{bmatrix}p_0\rb)\\
    & + \frac{q^3-1}{q-1}\sum_{\substack{x,r\in\FF_q}}f\lb(\begin{bmatrix}1&0&-x\\0&1&-r\\0&0&1\end{bmatrix}\gamma\begin{bmatrix}1&0&x\\0&1&r\\0&0&1\end{bmatrix}p_0\rb)\\
    & = (q^2+q+1) \sum_{\substack{x,y,r,s\in \FF_q\\ys\neq 1}}f\lb( 
\begin{bmatrix}
 \frac{b-a s y}{1-s y} & \frac{s (a-b)}{s y-1} & \frac{a s x-b r+c (r-s x)}{s y-1} \\
 \frac{y (b-a)}{s y-1} & \frac{a-b s y}{1-s y} & \frac{-a x+b r y+c (x-r y)}{s y-1} \\
 0 & 0 & c \\
      \end{bmatrix} p_0\rb)+\\
&+(q^2+q+1)\sum_{\substack{x,r,s\in\FF_q\\s\neq 0}}f\lb(\begin{bmatrix}a&0&ax-cx\\ bs-as&b&-axs+br+c\left(xs-r\right)\\ 0&0&c\end{bmatrix}p_0\rb)+\\
&+(q^2+q+1)\sum_{\substack{x,y,r\in\FF_q\\y\neq 0}}f\lb(\begin{bmatrix}a&ay-by&ax-byr+c\left(yr-x\right)\\ 0&b&br-cr\\ 0&0&c\end{bmatrix}p_0\rb)+\\
&+(q^2+q+1)\sum_{x,y\in\FF_q}f\lb(\begin{bmatrix}a&0&ax-cx\\ 0&b&br-cr\\ 0&0&c\end{bmatrix}p_0\rb)=(q^2+q+1)(S_1+S_2+S_3+S_4)
\end{align*}
We simplify each of the four terms separately. Observe that the substitution:

$$\begin{bmatrix}\tilde{x}\\\tilde{y}\end{bmatrix}=\begin{bmatrix}\frac{asx-br+c(r-sx)}{sy-1}\\\frac{-ax+bry+c(x-ry)}{sy-1}\end{bmatrix}=\frac{1}{sy-1}\begin{bmatrix}(a-c)s&c-b\\c-a&(b-c)y\end{bmatrix}\begin{bmatrix}x\\r\end{bmatrix}$$

is invertible as $\begin{vmatrix}(a-c)s&c-b\\c-a&(b-c)y\end{vmatrix}=(a-c)(b-c)(sy-1)\neq 0$. Therefore:

$$S_1=\sum_{x,r\in\FF_q}\sum_{\substack{y,s\in\FF_q\\ys\neq 1}}f\lb( 
\begin{bmatrix}
 \frac{b-a s y}{1-s y} & \frac{s (a-b)}{s y-1} & x \\
 \frac{y (b-a)}{s y-1} & \frac{a-b s y}{1-s y} & r \\
 0 & 0 & c \\
      \end{bmatrix} p_0\rb)=\sum_{x,r\in\FF_q}f\lb(\begin{bmatrix}
 b & 0 & x \\
 0 & a & r \\
 0 & 0 & c \\
      \end{bmatrix}p_0\rb)+$$
      
      $$+\sum_{\substack{x,y,r\in\FF_q\\y\neq 0}}f\lb(\begin{bmatrix}b&0&x\\y(a-b)&a&r\\0&0&c\end{bmatrix}p_0\rb)+\sum_{\substack{x,r,s\in\FF_q\\s\neq 0}}f\lb(\begin{bmatrix}b&s(b-a)&x\\0&a&r\\0&0&c\end{bmatrix}p_0\rb)+\sum_{x,r\in\FF_q}\sum_{\substack{y,s\in\FF_q^\times\\ys\neq 1}}f\lb( 
\begin{bmatrix}
 \frac{b-a s y}{1-s y} & \frac{s (a-b)}{s y-1} & x \\
 \frac{y (b-a)}{s y-1} & \frac{a-b s y}{1-s y} & r \\
 0 & 0 & c \\
      \end{bmatrix} p_0\rb)$$

We can make similar changes of variables $\begin{bmatrix}\tilde{x}\\\tilde{r}\end{bmatrix}$ as above and find that:

$$S_2=\sum_{\substack{x,r\in\FF_q\\s\in\FF_q^\times}}f\lb(\begin{bmatrix}a&0&x\\ (b-a)s&b&r\\ 0&0&c\end{bmatrix}p_0\rb)$$

$$S_3=\sum_{\substack{x,r\in\FF_q\\y\in \FF_q^\times}}f\lb(\begin{bmatrix}a&(a-b)y&x\\ 0&b&r\\ 0&0&c\end{bmatrix}p_0\rb)$$

$$S_4=\sum_{x,r\in\FF_q}f\lb(\begin{bmatrix}a&0&x\\ 0&b&r\\ 0&0&c\end{bmatrix}p_0\rb)$$

Putting together all the above we get that $S_1+S_2+S_3+S_4$ is the following sum:

$$S_1+S_2+S_3+S_4=Hf\lb(\begin{bmatrix}
  a/c&0\\0&b/c\end{bmatrix}\rb)$$

This is because each conjugate of $\big[\begin{smallmatrix}a&0\\0&b\end{smallmatrix}\big]$ appears exactly once in the upper $2\times 2$ block in the sum $S_1+S_2+S_3+S_4$.

Thus, the total contribution is:
  $$\frac{1}{6}(q^2+q+1) \frac{(q-1)^3}{(p-1)^3}\sum_{\substack{a,b,c\in \FF_p^{\times}\\a\neq b\neq c}}Hf\lb(\begin{bmatrix}
  a/c&0\\0&b/c\end{bmatrix}\rb)=\frac16(q^2+q+1)\frac{(q-1)^4}{(p-1)^3}\sum_{\substack{a,b\in\FF_p^\times\backslash\{1\}\\a\neq b}}Hf\lb(\begin{bmatrix}
  a&0\\0&b\end{bmatrix}\rb)=$$
  
  $$=\frac13(q^3-1)\frac{(q-1)^3}{(p-1)^3}\sum_{\substack{\{a,b\}\subseteq\FF_p^\times\backslash\{1\}}}Hf\lb(\begin{bmatrix}
  a&0\\0&b\end{bmatrix}\rb)$$

\item{\textbf{Proof of Proposition 5.1}}
    \begin{proof}
First to show the uniqueness of each representative, suppose there exists $m\in G_{\gamma}$ such that
       \[m\cdot (u_1\delta^{1/3},v_1\delta^{1/3}+\delta^{2/3}) = (u_2\delta^{1/3},v_2\delta^{1/3}+\delta^{2/3}).\]
       Then we deduce that $y = 0$ and so $u_1 = u_2$. We also see
       $c=d$, implying $v_1=v_2$. Next, suppose
       \[m\cdot (u_1 \delta^{1/3},v_1 \delta^{1/3}+\delta^{2/3}) =
         (\delta^{1/3}+u_2 \delta^{2/3}, \delta^{1/3}).\]
       But this immediately forces $c = 0$, a contradiction. Similarly,
       \[m\cdot (\delta^{1/3}+ u_1\delta^{2/3}, \delta^{1/3}) =(\delta^{1/3}+ u_2\delta^{2/3}, \delta^{1/3}),\]
       directly implies $u_1=u_2$.
       
      To show the completeness of this fundamental domain, take an arbitrary element $(\alpha,\beta) \in \HH_q$.  If $\beta_3 \neq 0$, then observe that
       \[\begin{bmatrix}
        1 &\alpha_3 & \alpha_1\\
        0 & \beta_3 & \beta_1\\
        0 & 0 & 1
      \end{bmatrix} \lb((\alpha_2-\alpha_3 \frac{\beta_2}{\beta_3})\delta^{1/3}, \frac{\beta_2}{\beta_3}\delta^{1/3}+\delta^{2/3}\rb) = (\alpha,\beta).\]
   Otherwise, we must have $\beta_2 \neq 0$, in which case
    \[\begin{bmatrix}
       1 &\alpha_2-1 & \alpha_1\\
        0 & \beta_2 & \beta_1\\
        0 & 0 & 1
      \end{bmatrix} \lb(\delta^{1/3}+\alpha_3\delta^{2/3}, \delta^{1/3}\rb) = (\alpha,\beta).\]
     \end{proof}

\item \textbf{First Parabolic Term.}
 \begin{align*}
    I_G(f,\gamma) &= \sum_{t\in G_{\gamma}\bk G/K} \sum_{u\in K/\{aI\}}  f((tu)^{-1}\gamma (tu))\\
                  & = \sum_{t\in G_{\gamma}\bk \HH_q} \frac{|K|}{|\{aI\}|} f(t^{-1}\gamma t)\\
    & = \frac{q^3-1}{q-1} \lb( \sum_{\substack{u,v\in \FF_q\\u\neq 0}} f\lb(
\begin{bmatrix}
 -\frac{v}{u} & 1 & 0 \\
 \frac{1}{u} & 0 & 0 \\
 0 & 0 & 1 \\
\end{bmatrix}
\gamma
\begin{bmatrix}
 0 & u & 0 \\
 1 & v & 0 \\
 0 & 0 & 1 \\
\end{bmatrix} p_0
                \rb) + \sum_{\substack{u\in \FF_q^{\times}}} f\lb(
\begin{bmatrix}
 \frac{1}{u} & -\frac{1}{u} & 0 \\
 0 & 1 & 0 \\
 0 & 0 & 1 \\
\end{bmatrix}
\gamma
\begin{bmatrix}
 u & 1 & 0 \\
 0 & 1 & 0 \\
 0 & 0 & 1 \\
\end{bmatrix} p_0\rb)\rb) \\
    & = (q^2+q+1)\lb(\sum_{\substack{u,v\in \FF_q\\u\neq 0}}f\lb( 
\begin{bmatrix}
 a & 0 & -\frac{av}{u} \\
 0 & a & \frac{a}{u} \\
 0 & 0 & a \\
      \end{bmatrix} p_0\rb) + \sum_{\substack{u\in \FF_q^{\times}}}f\lb( 
\begin{bmatrix}
 a & 0 & \frac{a}{u} \\
 0 & a & 0 \\
 0 & 0 & a \\
      \end{bmatrix} p_0\rb)\rb)\\
      & = (q^2+q+1) \lb(\sum_{\substack{u,v\in \FF_q\\u\neq 0}}f\lb(-\frac{v}{u}+\delta^{2/3},\frac{1}{u}+\delta^{1/3}\rb) + \sum_{\substack{u\in \FF_q^{\times}}} f\lb(\delta^{2/3}+\frac{1}{u},\delta^{1/3}\rb) \rb)\\
      &=(q^2+q+1)\lb(Hf(I)-f(p_0)\rb)
\end{align*}

\item \textbf{Proof of Proposition 5.2.}
First, we check that if $m\in B$ such that $mx=y$ for $x,y\in B\bk \HH_q$, then $x=y$. Suppose there exists $m\in B$ such that\[
      m(\delta^{1/3},v_1\delta^{1/3}+\delta^{2/3}) = (\delta^{1/3},v_2\delta^{1/3}+\delta^{2/3}).
      \]
       Then we right away see that $c/a=1$ and so $v_1=v_2$. Finally, it is also easy to see that we cannot use $B$ to move elements between the two sets in the disjoint union of $B\backslash \HH_q$.
       
      Next, we check that for any arbitrary element $(\alpha,\beta) \in \HH_q$ (where $\alpha=\alpha_1+\alpha_2\delta^{1/3}+\alpha_3\delta^{2/3}$ and $\beta=\beta_1+\beta_2\delta^{1/3}+\beta_3\delta^{2/3}$), there exists $m\in B$ and $x\in B\bk \HH_q$ such that $mx=(\alpha,\beta)$. First suppose that $\beta_3 \neq 0$.
      \[\begin{bmatrix}
        \alpha_2-\alpha_3\frac{\beta_2}{\beta_3} & \alpha_3 & \alpha_1\\
        0 & \beta_3 & \beta_1\\
        0 & 0 & 1
      \end{bmatrix} \lb(\delta^{1/3}, \frac{\beta_2}{\beta_3}\delta^{1/3}+\delta^{2/3}\rb) = (\alpha,\beta)\]
      If $\beta_3 = 0$, then\[
      \begin{bmatrix}
       \alpha_3 & \alpha_2 & \alpha_1\\
        0 & \beta_2 & \beta_1\\
        0 & 0 & 1
      \end{bmatrix} \lb(\delta^{2/3}, \delta^{1/3}\rb) = (\alpha,\beta)\]

\item \textbf{Second Parabolic Term.}
\begin{align*}
    I_G(f,\gamma) &= \sum_{t\in G_{\gamma}\bk G/K} \sum_{u\in K/\{aI\}}  f((tu)^{-1}\gamma (tu))\\
    &= \sum_{t\in G_{\gamma}\backslash\HH_q}\frac{|K|}{|aI|}f(t^{-1}\gamma t)\\
    &= \frac{q^3-1}{q-1}\lb(f\lb(\begin{bmatrix}
      1 & 0 & 0\\
      0 & 1 & 0\\
      0 & 0 & 1
    \end{bmatrix}\gamma \begin{bmatrix}
      1 & 0 & 0\\
      0 & 1 & 0\\
      0 & 0 & 1
    \end{bmatrix} p_0\rb)+\sum_{v\in \FF_q} f\lb(\begin{bmatrix}
      -v & 1 & 0\\
      1 & 0 & 0\\
      0 & 0 & 1
    \end{bmatrix}\gamma\begin{bmatrix}
      0 & 1 & 0\\
      1 & v & 0\\
      0 & 0 & 1
    \end{bmatrix}p_0\rb)\rb)\\
    &= (q^2+q+1)\lb(f\lb(\begin{bmatrix}
      a & a & 0\\
      0 & a & a\\
      0 & 0 & a
    \end{bmatrix}p_0\rb)+\sum_{v\in \FF_q} f\lb(\begin{bmatrix}-av+a&-av^2&a\\ a&av+a&0\\ 0&0&a\end{bmatrix}p_0\rb)\rb)\\
    &= (q^2+q+1)\lb(f\lb(\delta^{2/3}+\delta^{1/3},\delta^{1/3}+1\rb)+\sum_{v\in \FF_q} f\lb((1-v)\delta^{2/3}-v^2\delta^{1/3}+1,\delta^{2/3}+(v+1)\delta^{1/3}\rb)\rb)
\end{align*}

\item \textbf{Third Parabolic Term.}
\[G_\gamma\bk\HH_q = \{(\delta^{1/3}+u\delta^{2/3}+v, r\delta^{2/3}+s): r\neq 0\} \sqcup \{(\delta^{2/3}+v,r\delta^{1/3}+s): r\neq 0\}.\]

 \begin{align*}
    I_G(f,\gamma) &= \sum_{t\in G_{\gamma}\bk G/K} \sum_{u\in K/\{aI\}}  f((tu)^{-1}\gamma (tu))\\
                  & = \sum_{t\in G_{\gamma}\bk \HH_q} \frac{|K|}{|\{aI\}|} f(t^{-1}\gamma t)\\
    & = \frac{q^3-1}{q-1} \lb( \sum_{\substack{u,v,r,s\in \FF_q\\r\neq 0}} f\lb(
\begin{bmatrix}
 0 & \frac{1}{r} & -\frac{s}{r} \\
 1 & -\frac{u}{r} & \frac{s u}{r}-v \\
 0 & 0 & 1 \\
\end{bmatrix}
\gamma
\begin{bmatrix}
 u & 1 & v \\
 r & 0 & s \\
 0 & 0 & 1 \\
\end{bmatrix} p_0
                \rb) + \sum_{\substack{v,r,s\in \FF_q\\r\neq 0}} f\lb(
\begin{bmatrix}
 1 & 0 & -v \\
 0 & \frac{1}{r} & -\frac{s}{r} \\
 0 & 0 & 1 \\
\end{bmatrix}
\gamma
\begin{bmatrix}
 1 & 0 & v \\
 0 & r & s \\
 0 & 0 & 1 \\
\end{bmatrix} p_0
                \rb)\rb) \\
    \begin{split}
    & = (q^2+q+1)\Biggl(\sum_{\substack{u,v,r,s\in \FF_q\\r\neq 0}}f\lb( 
\begin{bmatrix}
 \frac{a (r+u)}{r} & \frac{a}{r} & \frac{a (s+v)-b s}{r} \\
 -\frac{a u^2}{r} & \frac{a(r-u)}{r} & -\frac{-a r v+a s u+a u v+b r v-b s u}{r} \\
 0 & 0 & b \\
      \end{bmatrix} p_0\rb) 
      \\ &\qquad + \sum_{\substack{v,r,s\in \FF_q\\ r\neq 0}}f\lb(
\begin{bmatrix}
  a & 0 & v (a-b) \\
 \frac{a}{r} & a & \frac{a (s+v)-b s}{r} \\
 0 & 0 & b \\
      \end{bmatrix} p_0\rb)\Biggr).
 \end{split}
\end{align*}

Employ the following change of variables in the two sums:

$$\begin{bmatrix}\tilde{v}\\ \tilde{s}\end{bmatrix}=\frac{1}{r}\begin{bmatrix}
  a&(a-b)\\(a-b)r-au&(b-a)u
\end{bmatrix}\begin{bmatrix}
  v\\s
\end{bmatrix}$$

$$\begin{bmatrix}
  v^*\\s^*
\end{bmatrix}=\begin{bmatrix}
  a-b&0\\
  a/r&(a-b)/r
\end{bmatrix}\begin{bmatrix}
 v\\s
\end{bmatrix}$$

The determinants are both $(a-b)^2/r\neq 0$ hence the transformations are invertible. It follows that the above sums reduce to horocycle transforms of parabolic conjugacy classes in $\GL_2(\FF_q)$.

$$I_G(f,\gamma)=(q^2+q+1)Hf\lb(\begin{bmatrix}
  a/b&1\\0&a/b
\end{bmatrix}\rb)$$

\item \textbf{Proof of Proposition 6.1.}

 We first show that every $G_\gamma$-orbit on $\HH_q$ contains at least one element of the above form. Let $z\in\HH_q$ and write:
 
 $$z=\begin{bmatrix}
   a&b&c\\d&e&f\\0&0&1
 \end{bmatrix}p_0$$

We want to find $r,s\in\FF_q$ not both zero and $t\in\FF_q^\times$ such that:

$$\begin{bmatrix}
  r&s\xi&0\\s&r&0\\0&0&t
\end{bmatrix}\begin{bmatrix}
  a&b&c\\d&e&f\\0&0&1
\end{bmatrix}=\begin{bmatrix}v&u&x\\0&1&y\\0&0&1\end{bmatrix}$$

for some $v\in\FF_q^\times$ and $u,x,y\in\FF_q$. This amounts to solving the following set of equations for $s,r,t$ over $\FF_q$:

\begin{align*}
    sa+rd&=0\\
    sb+re&=1\\
    t&=1
\end{align*}

By assumption, the determinant $\begin{vmatrix}a&b\\d&e\end{vmatrix}$ is nonzero, therefore the above system admits a unique solution for $s,r,t$ over $\FF_q$, moreover, $s$ and $r$ cannot be simultaneously zero.

We now need to show that the domain does not contain any orbit repetitions. Suppose that:

$$\begin{bmatrix}
  r&s\xi&0\\s&r&0\\0&0&t
\end{bmatrix}\begin{bmatrix}v&u&x\\0&1&y\\0&0&1\end{bmatrix}=\begin{bmatrix}v'&u'&x'\\0&1&y'\\0&0&1\end{bmatrix}$$

Writing up the above equations again we get:

\begin{align*}
    sv&=0\\
    su+r&=1\\
    t&=1
\end{align*}

As $v\neq 0$ we find that $s=0$, $r=1$ and $t=1$ so that the first matrix on the left-hand side is the identity, which implies that $x=x'$, $y=y'$, $u=u'$ and $v=v'$, concluding the proof.

\item \textbf{Second Elliptic Term.}

$$I_G(f,\gamma)=\frac{(q-1)(q^2-1)}{(p-1)(p^2-1)}\cdot\frac{q^3-1}{q-1}\sum_{\substack{x,y,u\in\FF_q\\v\in\FF_q^\times}}f\lb(\begin{bmatrix}
  v&u&x\\0&1&y\\0&0&1
\end{bmatrix}^{-1}\begin{bmatrix}
  k&l\xi&0\\l&k&0\\0&0&m
\end{bmatrix}\begin{bmatrix}
  v&u&x\\0&1&y\\0&0&1
\end{bmatrix}p_0\rb)=$$

$$=\frac{(q^3-1)(q^2-1)}{(p-1)(p^2-1)}\sum_{\substack{u\in\FF_q\\v\in\FF_q^\times}}\sum_{x,y\in\FF_q}f\lb(\begin{bmatrix}-lu+k&\frac{l\xi-lu^2}{v}&\frac{x\left(-lu+k\right)+y\left(l\xi-ku\right)+m\left(uy-x\right)}{v}\\ vl&lu+k&lx+ky-my\\ 0&0&m\end{bmatrix}p_0\rb)$$

Set:

\begin{align*}
    \tilde{x}&=\frac{(k-m-lu)}{v}\cdot x+\frac{(m-k)u+l\xi}{v}\cdot y\\
    \tilde{y}&=lx+(k-m)y
\end{align*}

The determinant corresponding to this change of variables is $\lb((k-m)^2-l^2\xi\rb)/v$. It could only be zero if $k-m=0$ and $l=0$ since $\xi$ is not a square, but then $\gamma$ would be similar to a diagonal matrix, contradicting our assumption. Thus the orbital sum can be simplified as:

$$I_G(f,\gamma)=\frac{(q^3-1)(q^2-1)}{(p-1)(p^2-1)}\sum_{\substack{u\in\FF_q\\v\in\FF_q^\times}}\sum_{\tilde{x},\tilde{y}\in\FF_q}f\lb(\begin{bmatrix}-lu+k&\frac{l\xi-lu^2}{v}&\tilde{x}\\ vl&lu+k&\tilde{y}\\ 0&0&m\end{bmatrix}p_0\rb)$$

Finally, notice that the upper-left $2\times 2$ block is similar to the elliptic matrix $\Big[\begin{smallmatrix}k&l\xi\\l&k\end{smallmatrix}\Big]$ over $\GL_2(\FF_q)$. Hence we obtain
\[I_G(f,\gamma)=\frac{(q^3-1)(q^2-1)}{(p-1)(p^2-1)}Hf\lb(\begin{bmatrix}
  k/m&l\xi/m\\l/m&k/m
\end{bmatrix}\rb).\]

\end{itemize}


\end{document}